\documentclass[10pt,draftcls,onecolumn]{IEEEtran}

\usepackage[colorlinks,bookmarksopen,bookmarksnumbered,citecolor=red,urlcolor=red]{hyperref}
\usepackage{epsfig} 
\usepackage{subfig}
\usepackage{amsmath} 
\usepackage{amssymb}  
\usepackage{enumerate}
\usepackage{color}

\newtheorem{theorem}{Theorem}
\newtheorem{lemma}[theorem]{Lemma}

\newtheorem{corollary}[theorem]{Corollary}
\newtheorem{remark}{Remark}
\newtheorem{definition}{Definition}

\newtheorem{assumption}{Assumption}
\newtheorem{proof}{Proof}

\newcommand\BibTeX{{\rmfamily B\kern-.05em \textsc{i\kern-.025em b}\kern-.08em
T\kern-.1667em\lower.7ex\hbox{E}\kern-.125emX}}

\title{\LARGE \bf Feedback Controller Sparsification for a Class of Linear Systems with Parametric Uncertainties}  
\author{Reza Arastoo$^\dagger$  and MirSaleh Bahavarnia$^\dagger$
\thanks{$\dagger$ R. Arastoo and M. Bahavarnia are with the Department  of
  Mechanical Engineering  and  Mechanics,  Packard  Lab.,  Lehigh
  University,  Bethlehem, PA. Email Addresses: {\tt\small
    \{reza.arastoo,msbmo2008\}@gmail.com}}  }

\begin{document}
\maketitle
\thispagestyle{empty}
\pagestyle{plain}

\begin{abstract}
\noindent We consider the problem of output feedback controller sparsification for systems with parametric uncertainties. We develop an optimization scheme that minimizes the performance deterioration caused by the sparsification process, while enhancing sparsity pattern of the feedback gain. In order to improve temporal proximity of an existing closed-loop system and its sparsified counterpart, we also incorporate an additional constraint into the problem formulation so as to bound the variation in the system output pre and post sparsification. We also show that the resulting non-convex optimization problem can be equivalently reformulated into a rank-constrained optimization problem. We then formulate a bi-linear minimization program along with an iterative algorithm to obtain a sub-optimal solution which satisfies the rank constraint with arbitrary tolerance. Lastly, a sub-optimal sparse controller design for IEEE 39-bus New England power network is utilized to showcase the effectiveness of our proposed method.
\end{abstract}




\section{Introduction}
\vspace{-2pt}
In control theory, there has always been a desire to achieve the best possible performance, while taking into account the feasibility and cost of the communication between the subsystems. The reason behind such yearning is mostly rooted in the fact that unlike small-scale dynamical systems, where centralized control methodologies can efficiently be applied owing to the availability of information from all subsystems, the subsystem level information is not globally accessible throughout the network in medium to large scale systems. With the emergence and growth of ultra large-scale interconnected systems, such as power grids, transportation systems, and wireless data networks, exploiting the characteristics of the underlying structure of the system in the control design has become inescapable. Hence, the concepts of distributed and decentralized controllers have received increasing attention in recent years \cite{Polyak:2013,Bamieh:2002,Lavaei:2013,Lin:2011}.

It has been known that optimal controller design under controller structural constraints is a challenging problem. Nonetheless, numerous studies have been carried out to either propose controller design frameworks or reveal inherent structural properties of controllers for special classes of systems \cite{Rotkowitz:2006,Wang:2016ACC1,Wang:2016ACC2,Bamieh:2005}. Another concern in the design of large-scale control systems is the number of communication links between the subsystems which poses major issues especially when establishing links between nodes is very costly. Sparsifying the controller gain leads to fewer information pathways as well as fewer controller sensors and actuators. As a result, the design of controller gains with minimum number of non-zero entries can mitigate the communication overflow issues emergent in large interconnected systems. In the sparsity-promoting control problem, the ultimate objective is to minimize the number of feedback links without losing much performance. This is achieved by incorporating additional functions into the optimization cost function to penalize the number of communication links. The problem has been addressed by a number of researchers, who opted for various techniques to tackle the inherently non-convex problem \cite{Polyak:2013,Lin:2013,Arastoo:2014,Bahavarnia:2015,Lin:2012,Wytock:2013}.

For example, in \cite{Lin:2013,Lin:2012} the Alternating Direction Method of Multipliers is utilized to handle the non-convex terms in the problem formulation. In \cite{Arastoo:2014}, the authors proposed a novel framework in which all non-convexities are lumped into a rank constraint further enabling it to address output feedback problems with norm constraints on the input/output signals. In some recent papers an unconventional approach to synthesize near optimal sparse controllers has been adopted. This proposed sparse controller design framework is founded based on the assumption that a pre-designed well-performing controller is available and the ultimate goal is to obtain a sparse feedback controller approximating the attributes and qualities of the original well-preforming controller \cite{Arastoo:2016CDC,Arastoo:2015,Fattahi:2015CDC,Fattahi:2016Allerton}. 

In this paper, we extend the work published in \cite{Arastoo:2016CDC,Arastoo:2015}. In \cite{Arastoo:2015}, deterministic linear systems are considered and it is shown that the optimal sparse control problem for this class of systems can be solved by applying an ADMM algorithm to its equivalent rank-constrained optimization problem.  The work published in \cite{Arastoo:2016CDC} expanded the applicability of the method introduced in \cite{Arastoo:2015} to linear systems with parametric time-varying uncertainties. We showed that by utilizing $\mathcal{H}_2$ and $\mathcal{H}_\infty$ control theorems \cite{Xie:1992h,Xie:1992}, this novel approach can be easily applied to design sparse robust controllers by equivalently reformulating the original problem into a rank-constrained optimization where all non-convexities are collected into a rank constraint. The next notable improvement in \cite{Arastoo:2016CDC} is that the ADMM algorithm, used in \cite{Arastoo:2015} with no proof of convergence, is replaced by our novel convergent algorithm, which employs a bi-linear optimization to reach the sub-optimal solution of the rank-constrained optimization problem. In this paper, we provide detailed proof of all lemmas and theorems presented in \cite{Arastoo:2016CDC} without proofs. We also show that the optimization parameters can be tuned such that the optimal solution of our proposed minimization problem satisfies the rank constraint with arbitrary tolerance. Furthermore, we present discussions on the convergence of the proposed solving algorithm and provide remarks on the choice of the weight matrix to improve the sparsification performance without adversely affecting the convergence of the algorithm. The simulation section is also substantially extended to provide more vivid insights into advantages of applying the proposed procedure to synthesize sparse controllers for power networks. It further examines in more detail the impact of the structure and magnitude of the power network uncertainties on the robustness of the closed-loop systems as well as the sparsity level of the controllers.

This paper is structured as follows. Section \ref{Notas} provides key definitions and notations used throughout the paper. In Section \ref{sec:formulation}, we formally state the problem we aim to solve. In Sections \ref{sec:equ_ref} and \ref{sec:fix_rank_ref}, we elaborate how our problem can equivalently be reformulated into an optimization problem constrained to several linear matrix inequalities and a fixed rank constraint. Section \ref{sec:algorithm} provides insight into our proposed algorithm and states several related results. The results of our numerical simulations are presented in Section \ref{sec:simul}. Finally, we end with concluding remarks in Section \ref{sec:conc}.
\vspace{-6pt}
\section{Notations} \label{Notas}
\vspace{-2pt}
Throughout this paper, matrices are customarily referred to with capital letters, and the entries are referred to with the corresponding lower-case letters, using subscripts. The vectors, on the other hand, are symbolized by lower-case letters with components denoted by the same letter using subscripts. A unit vector with its $i^{th}$ entry equal to one is denoted by $e_i$. The set of real numbers is denoted by $\mathbb{R}$. The space of $n$ by $m$ matrices with real entries is indicated by $\mathbb{R}^{n\times m}$. The set of real matrices with non-negative (positive) entries is represented by $\mathbb{R}^{n\times m}_+$ ($\mathbb{R}^{n\times m}_{++}$). The $n$ by $n$ identity matrix is denoted by $I_{n}$. The vector of singular values of matrix $X$ is denoted by $\sigma(X)$.
The entry-wise product of two matrices, i.e., Hadamard product, is represented by $\circ$. If $X=\left[x_{ij}\right]$, then the matrix $|X|$ is the entry-wise absolute value of $X$, i.~e. $|X|=\left[|x_{ij}|\right]$. The number of non-zero entries of a matrix is denoted by $\| . \|_0$ while $\|.\|_1$ denotes $\ell_1$ norm, and the $\mathcal{L}_2$-norm is defined by 
\begin{align*}
\|x\|_{\mathcal{L}_{2}(\mathbb{R}^n)} := \Big(\int_0^{\infty}\|x(t)\|^2_2~dt\Big)^{1/2}.
\end{align*}
Whenever it is not confusing, we use $\mathcal{L}_2$ instead of $\mathcal{L}_{2}(\mathbb{R}^n)$.
$\mathbf{Tr(.)}$ and $\mathbf{rank}(.)$ denote the trace and rank of the matrix operands, respectively. The operator $\mathbf{diag}(.)$ constructs block diagonal matrix from input arguments.
\begin{definition}\label{rank-tol}
For a given $\epsilon>0$ and matrix $X$, we say that rank of $X$ is $k$ with tolerance $\epsilon$, and it is denoted by  $\mathbf{rank}(X;\epsilon)$, if exactly $k$ singular values of $X$ are larger than or equal to $\epsilon$.
\end{definition}
 A matrix is said to be Hurwitz if all of its eigenvalues lie within the open left half of the complex plane. A real symmetric matrix is said to be positive definite (semi-definite) if all of its eigenvalues are positive (non-negative). $\mathbb{S}^n_{++}$ ($\mathbb{S}^n_{+}$) denotes the space of positive definite (positive semi-definite) real symmetric matrices, and the notation $X\succeq Y$ ($X\succ Y$) means $X-Y\in \mathbb{S}^n_+$ ($X-Y\in \mathbb{S}^n_{++}$). 
\begin{remark}
For simplicity of our notations, we will use a new notation in statements of theorems, where we use an asterisk '*' to represent the upper triangular sub-blocks of symmetric matrices. Moreover, in the occasions when the optimal solutions of the optimization problems in these theorems do not depend on some of the sub-blocks of matrices, we use a dash '-' to represent such sub-blocks with no apparent utilization in the problem.
\end{remark}
\vspace{-6pt}
\section{Problem Formulation}\label{sec:formulation}
\vspace{-2pt}
\subsection{LTI Systems with Parametric Uncertainties}
The focus of this paper is on the following class of uncertain linear time-invariant (LTI) systems that are defined by the state space realization\footnote{It is assumed that the pair $(A,B_1)$ is controllable and $(A,C)$ is detectable.}
\begin{align}\label{eq:system}
\dot{x}(t)&= [A+\Delta_A]x(t)+[B_1+\Delta_{B_1}]u(t)+B_2d(t) \notag\\
y(t)&=C x(t),
\end{align}
where $x(t)\in \mathbb{R}^n$ is the state vector, $u(t)\in \mathbb{R}^{m}$ is the control input and $d(t)\in \mathbb{R}^{p}$ represents the exogenous disturbance input. We assume that the matrices $A\in\mathbb{R}^{n \times n}$, $B_1\in \mathbb{R}^{n \times m}$, $B_2\in \mathbb{R}^{n \times p}$, and $C\in \mathbb{R}^{q \times n}$  are constant real matrices describing the dynamics of the nominal system, whereas $\Delta_A$ and $\Delta_{ B_1}$ represent the parameter uncertainties of the matrices $A$ and $B_1$, respectively. In this paper, we consider a special uncertainty structure expressed by 
\begin{align}\label{eq:uncert}
\left[\begin{array}{cc} \Delta_A & \Delta_{B_1}\end{array}\right]=D\Delta \left[\begin{array}{cc} E_{A}&E_{B_1}\end{array}\right],
\end{align}
where $D$, $E_A$ and $E_{B_1}$ are known constant real matrices with appropriate dimensions, which characterize the structure of the uncertainties, while $\Delta$ is an unknown $i$ by $j$ real matrix which is constrained by
\begin{align} \label{eq:uncertainty}
\Delta^{\text T}\Delta \preceq \rho^2 I_j.
\end{align}
This class of uncertain linear systems was initially reported by Petersen in papers \cite{Petersen:1987,Petersen:1988}
and later thoroughly addressed by Khargonekar {\emph et al.} \cite{Khargonekar:1990}.
\subsection{Controller Sparsification via \texorpdfstring{$\mathcal{H}_p$}{Hp} Approximations}
Suppose that a pre-designed well-performing controller, namely ${\hat{K}}$, is readily available and the nominal system controlled by such a controller, represented by $\hat{\mathcal{S}}$, has all the desired characteristics. 
The objective is to synthesize a constant gain output feedback controller of the form
\begin{align}\label{eq:controller}
u(t)=Ky(t), \:\:\: K\in \mathcal{K}
\end{align}
with minimum number of non-zero entries, while minimizing the performance deterioration from that of the closed-loop system $\hat{\mathcal{S}}$ under parametric uncertainties. In (\ref{eq:controller}), $\mathcal{K}$ denotes a set of admissible feedback gains which holds desirable properties such as pre-defined communication layout. 
\begin{assumption}
It is assumed that the set $\mathcal{K}$ is convex.
\end{assumption}
It should be emphasized that this assumption does not offer any premise on characterization of the set of all stabilizable output feedback controllers. In our follow-up discussions, we will show that if our proposed optimal control design is feasible, then the resulting output feedback controller will be stabilizing and satisfy the structural constraint $K \in \mathcal{K}$. 
There are numerous applications associated with such convexly constrained controller design, such as power grids or multi-UAV systems. It is sometimes practically infeasible to establish some specific communication links between particular nodes due to the nodes distant locations or security issues in networks. There are also cases where the attenuation/amplification in certain feedback paths is upper bounded, due to technological shortcomings. Such restrictions are addressed by forcing the corresponding controller entries to be contained in a convex set.

Our goal is to solve the following $\ell_0$-regularized optimal control problem to compute a sparse output feedback controller under parametric uncertainties 
\begin{subequations}\label{eq:OSP}
\begin{alignat}{3}
& \underset{K,\varepsilon_y,\varepsilon_{\mathcal{S}}}{\textrm{minimize}}~~&&
\varepsilon_{\mathcal{S}}+\lambda_1 \varepsilon_y +\lambda_2 \|{K}\|_{0}\hspace{80pt}&\\
& \mbox{\textrm{subject to}}~~&&K\in\mathcal{K}&\label{eq:OSP_b}\\
& &&\mathcal{S:}~\mbox{Stable}&\label{eq:OSP_c}\\
& &&\|y_{\mathcal{S}}-y_{\hat{\mathcal{S}}}\|_{\mathcal{L}_2}<\varepsilon_y\|d\|_{\mathcal{L}_2}&\label{eq:OSP_d}\\
& && \|\mathcal{S}-\hat{\mathcal{S}}\|^2_{{\mathcal{H}_2}}<\varepsilon_{\mathcal{S}}&\label{eq:OSP_e}
\end{alignat}
\end{subequations}
in which $\|.\|_{{\mathcal{H}_2}}$ is 
the well-known $\mathcal{H}_2$ norm. Nominal closed-loop system $\hat{\mathcal{S}}$ is a previously designed desired optimal closed-loop system with output signal $y_{\hat{\mathcal{S}}}$ and $\mathcal{S}$ is the resulting system by closing the loop using sparse feedback controller $K$. The output signal of $\mathcal{S}$ is denoted by $y_{\mathcal{S}}$. In order to promote sparsity of feedback gain matrix $K$, the $\ell_0$  measure of $K$, which is denoted by $\|{K}\|_{0}$, has been added to the cost function. Two design parameters $\lambda_1$ and $\lambda_2$ are introduced to achieve  desired trade-off between performance loss and sparsity. 

In the optimal control problem (\ref{eq:OSP}),  constraint (\ref{eq:OSP_e}) is included to ensure that the nominal closed-loop system $\hat{\mathcal{S}}$ is well-approximated by a closed-loop system controlled by a sparse controller $K$. To enhance temporal features of our approximation, we also incorporate another requirement into our design scheme, characterized by constraint (\ref{eq:OSP_d}). This constraint guarantees that the energy level of the difference between the output signals of the two closed-loop systems remains under a pre-specified level $\epsilon_y$ when both closed-loop systems are excited by a disturbance input $d$ with unit norm. 

The goal of this paper is to study the effect of parametric uncertainties on the best achievable levels of sparsity. However, finding the optimal solution of the problem (\ref{eq:OSP}) is inherently NP-hard; see our discussion in Section \ref{sec:fix_rank_ref}. In Section \ref{sec:algorithm}, we will propose a tractable approximation algorithm to solve this problem. The following sections discuss the equivalent problem reformulation exploited in numerically solving our optimization problem.
\vspace{-6pt}
\section{Equivalent Reformulation}\label{sec:equ_ref}
\vspace{-2pt}
The first two terms in the cost function of the optimization problem (\ref{eq:OSP}) can be simplified into the $\mathcal{H}_2$/$\mathcal{H}_\infty$  norms of an augmented system, namely $\bar{\mathcal{S}}$, constructed by the following state space realization matrices
\begin{align}\label{eq:mat_bar}
\bar{A}&=\mathbf{diag}(\bar{A}_{11},A+B_1\hat{K}C),\\
\bar{B}&=\left[\begin{array}{cc} B_2^{\text T} &B_2^{\text T} \end{array}\right]^{\text T} ,~~
\bar{C}=\left[\begin{array}{cc} C&-C\end{array}\right],\notag
\end{align}
where $\bar{A}_{11}=[A+\Delta_A]+[B_1+\Delta_{B_1}]KC$. As it can be seen, the system $\bar{\mathcal{S}}$ represents the difference between the nominal system controlled by the pre-designed controller and the uncertain system, stabilized by closing its feedback loop using a sparse controller. Hence, we can re-formulate our problem into the $\mathcal{H}_2$/$\mathcal{H}_\infty$ norm minimization of the augmented system as follows: 
\begin{alignat}{3}
&\underset{K,\varepsilon_{y},\varepsilon_{\mathcal{S}}}{\textrm{minimize}}~\underset{\Delta_A,\Delta_{B_1}}{\textrm{maximize}}~~&& \varepsilon_{\mathcal{S}}+ \lambda_1 \varepsilon_y+\lambda_2 \|K\|_{0} \hspace{45pt}\label{eq:H2HiOSP}\\
&\mbox{\textrm{subject to}}~~&& K\in\mathcal{K},\notag\\
& &&\bar{A}_{11}~\mbox{Hurwitz},\notag\\
& &&\|\bar{C}(sI-\bar{A})^{-1}\bar{B}\|_{\mathcal{H}_\infty}<\varepsilon_y\notag,\\
& &&\| \bar{C}(sI-\bar{A})^{-1}\bar{B} \|^{2}_{\mathcal{H}_2} <\varepsilon_{\mathcal{S}}.\notag
\end{alignat}
In problem (\ref{eq:H2HiOSP}), the attempt is to minimize the worst case gap between the frequency response of the systems in the sense of a weighted sum of  the $\mathcal{H}_2$ and $\mathcal{H}_\infty$ norms. Therefore, unlike the design schemes introduced in \cite{Lin:2013,Arastoo:2014}, the approach proposed in this paper allows us to exploit the advantages offered by other controller design schemes in the sparse controller design.   
In the next section, we show that the optimization problem (\ref{eq:H2HiOSP}) includes bi-linear matrix inequality constraints mainly due to the existence of the Lyapunov stability conditions. Here, we intend to employ the idea of lumping all nonlinear constraints into a rank-constrained problem, proposed in \cite{Arastoo:2014}, to rewrite problem as a rank-constrained optimization. Based on the obtained reformulation, it is possible to either develop heuristics to sub-optimally solve the problem or provide necessary and sufficient conditions for the feasibility of the points with particular desired costs. 
\vspace{-6pt}
\section{Fixed Rank Optimization Reformulation}\label{sec:fix_rank_ref}
\vspace{-2pt}
The approach adopted in this paper is based on solving the problem of sparse controller approximation via rank-constrained optimization. Hence, we start by stating the main lemmas which helps us cast the constraints of the optimization problem as rank-constrained linear matrix inequalities.
\begin{lemma}[\cite{Arastoo:2014}] \label{lem:non_to_rank}
Let $\mathcal{U} \in \mathbb{R}^{n \times n}$, $\mathcal{V} \in \mathbb{R}^{n \times m}$, $\mathcal{W} \in \mathbb{R}^{m \times m}$, and $\mathcal{Y} \in \mathbb{R}^{m \times n}$, with  $\mathcal{U} \succ 0$.
Then, $\mathbf{rank}(\mathcal{M})=n$ if and only if $\mathcal{W}=\mathcal{YUY}^{\text T}$, $\mathcal{V}^{\text T}=\mathcal{YU}$, and $\mathcal{Z}=\mathcal{U}^{-1}$ where
\begin{align*}
\mathcal{M}=\left[\begin{array}{ccc} \mathcal{U} & \mathcal{V} & I_n \\ \mathcal{V}^{\text T} & \mathcal{W} & \mathcal{Y} \\ I_{n}& \mathcal{Y}^{\text T} & \mathcal{Z} \end{array}\right].
\end{align*}
\end{lemma}
The above lemma can be utilized to collect almost all non-convex terms of the optimization problems in one and only one constraint in the form of a rank constraint. There are a number of algorithms proposed to solve rank-constrained optimization problems \cite{Mishra:2013,Yu:2011,Shai:2011,Gao:2010,Absil:2015}. In this manuscript, we aim to render such algorithms applicable in solving our inherently nonlinear controller sparsification problem by collecting various forms of non-convex/combinatorial constraints into a fixed rank constraint. 

As a first step, we show how the $\mathcal{H}_2$ norm of an uncertain system can be formulated by rank-constrained linear matrix inequalities.
\begin{lemma}\label{thm:H2_to_rank_uncert}
Given a strictly proper uncertain linear system $\mathcal{P}$ with state space realization $(\mathcal{A}+\Delta_{\mathcal{A}},\mathcal{B},\mathcal{C})$, where  $\mathcal{A}\in\mathbb{R}^{n \times n}$, $\mathcal{B}\in \mathbb{R}^{n \times m}$, $\mathcal{C}\in \mathbb{R}^{q \times n}$, $\Delta_{\mathcal{A}}=\mathcal{D}\Delta \mathcal{E}$ and $\Delta^{\text T}\Delta\preceq\rho^2 I_j$, then $\mathcal{P}$ is stable and $\|\mathcal{P}\|_{\mathcal{H}_2}^2<\gamma$ if and only if there exists a positive definite matrix $\mathcal{X}\succ0$ and a positive scalar $\varepsilon$ such that
\begin{align*}
\mathbf{Tr}(\mathcal{CX}\mathcal{C}^{\text T})<\gamma,\\
\left[\begin{array}{cc}
\mathcal{Y}_1+\mathcal{Y}_1^{\text T}+\mathcal{BB}^{\text T}+\varepsilon {\rho} \mathcal{DD}^{\text T} & \sqrt{\rho} \mathcal{Y}_2\\
\sqrt{\rho} \mathcal{Y}_2^{\text T} & -\varepsilon I_j \end{array}\right] \prec 0,\\
\mathbf{rank}\left[\begin{array}{ccccc} \mathcal{X} & *&  * & *\\
\mathcal{Y}_1^{\text T} & - &* &*\\
\mathcal{Y}_2^{\text T} & - & -& *\\
I_n & \mathcal{A}^{\text T} & \mathcal{E}^{\text T}&-
\end{array} \right]=n.
\end{align*}
\end{lemma}
\begin{proof}
The system $\mathcal{P}$ is stable with $\mathcal{H}_2$ norm less than $\gamma$ if and only if there exists a positive definite matrix $\mathcal{X}$ such that \cite[p.~210]{Dullerud:2000}
\begin{align*}
\mathbf{Tr}(\mathcal{CXC}^{\text T})<\gamma,\\
(\mathcal{A}+\Delta_{\mathcal{A}})\mathcal{X}+\mathcal{X}(\mathcal{A}+\Delta_{\mathcal{A}})+\mathcal{BB}^{\text T}\prec 0.
\end{align*}
Substituting $\Delta_{\mathcal{A}}=\mathcal{D}\Delta \mathcal{E}$ into the second equation, we have
\begin{align*}
\mathcal{AX}+\mathcal{XA}^{\text T}+\mathcal{BB}^{\text T}+\mathcal{D}\Delta \mathcal{EX}+\mathcal{X}(\mathcal{D}\Delta \mathcal{E})^{\text T}\prec 0.
\end{align*}
Since the term $\mathcal{AX+XA^{\text T}+BB^{\text T}}$ is symmetric and $\Delta^{\text T}\Delta\preceq \rho^2 I_j$, the above linear matrix inequality holds if and only if there exists a positive scalar $\varepsilon>0$ such that \cite{Xie:1992h}
\begin{align*}
\mathcal{AX}+\mathcal{XA}^{\text T}+\mathcal{BB}^{\text T}+\varepsilon {\rho} \mathcal{DD}^{\text T}+\varepsilon^{-1}\rho \mathcal{XE}^{\text T}\mathcal{EX}\prec 0.
\end{align*}
This is equivalent to having 
\begin{align*}
\left[\begin{array}{cc}
\mathcal{AX}+\mathcal{XA}^{\text T}+\mathcal{BB}^{\text T}+\varepsilon {\rho} \mathcal{DD}^{\text T} & \sqrt{\rho}(\mathcal{EX})^{\text T}\\
\sqrt{\rho} (\mathcal{E X}) & -\varepsilon I_j \end{array}\right] \prec 0.
\end{align*}
Applying Lemma \ref{lem:non_to_rank}, the last LMI can equivalently be rewritten as shown below.
\begin{align*}
\left[\begin{array}{cc}
\mathcal{Y}_1+\mathcal{Y}_1^{\text T}+\mathcal{BB}^{\text T}+\varepsilon {\rho} \mathcal{DD}^{\text T} & \sqrt{\rho} \mathcal{Y}_2\\
\sqrt{\rho} \mathcal{Y}_2^{\text T} & -\varepsilon  I_j\end{array}\right] \prec 0,\\
\mathbf{rank}\left[\begin{array}{ccc} \mathcal{X} & \mathcal{Y}_1 & \mathcal{Y}_2\\I_n & \mathcal{A}^{\text T} & \mathcal{E}^{\text T}\end{array} \right]=n.
\end{align*}
Augmenting proper rows and columns to the rank-constrained matrix to make it symmetric completes our proof.
\end{proof}
Similar to Lemma \ref{thm:H2_to_rank_uncert}, which paves the way in casting the $\mathcal{H}_2$ norm term in our optimal controller sparsification problem, as a rank-constrained optimization problem, the $\mathcal{H}_{\infty}$ norm term of problem (\ref{eq:H2HiOSP}) can also be equivalently represented with a set of rank-constrained linear matrix inequalities. In the next lemma, we prove such equivalence, which later helps in accommodating the whole problem of controller sparsification under parametric uncertainties into the framework of rank-constrained optimization.
\begin{lemma}\label{thm:Hi_to_rank_uncert}
Suppose a strictly proper uncertain LTI plant $\mathcal{P}$, represented in the state space triplet $(\mathcal{A}+{\Delta}_{\mathcal{A}},\mathcal{B},\mathcal{C})$, where  $\mathcal{A}\in\mathbb{R}^{n \times n}$, $\mathcal{B}\in \mathbb{R}^{n \times m}$, $\mathcal{C}\in \mathbb{R}^{q \times n}$, ${\Delta}_{\mathcal{A}}=\mathcal{D}{\Delta} \mathcal{E}$ and ${\Delta}^{\text T}{\Delta}\preceq\rho^2 I_j$, then the system is stable with $\mathcal{H}_\infty$ norm less than $\gamma$ if and only if there exists a positive definite matrix $\mathcal{X}\succ0$ and a positive scalar $\varepsilon>0$ satisfying
\begin{align*}
\left[\begin{array}{cccc}
\mathcal{Y}_1+\mathcal{Y}_1^{\text T}+\varepsilon {\rho} \mathcal{D}\mathcal{D}^{\text T} &*&*&*\\
\mathcal{B}^{\text T}& -\gamma I_m& *& *\\
(\mathcal{CX})&0&-\gamma I_q&*\\
\sqrt{\rho}\mathcal{Y}_2^{\text T} & 0 & 0 &-\varepsilon I_j
\end{array}\right] \prec 0,\\
\mathbf{rank}\left[\begin{array}{ccccc} \mathcal{X} & *&  * & *\\
\mathcal{Y}_1^{\text T} & - &* &*\\
\mathcal{Y}_2^{\text T} & - & -& *\\
I_n & \mathcal{A}^{\text T} & \mathcal{E}^{\text T}&-
\end{array} \right]=n.
\end{align*}
\end{lemma}
\begin{proof}
Employing Lemma 7.4 in \cite[p.~221]{Dullerud:2000}, the plant $\mathcal{P}$ is stable with $\|\mathcal{P}\|_\infty<\gamma$ if and only if there exists a positive definite matrix $\mathcal{Z}$ such that
\begin{align*}
\mathcal{Z}(\mathcal{A}+\Delta_{\mathcal{A}})+(\mathcal{A}+\Delta_{\mathcal{A}})^{\text T}\mathcal{Z}+\mathcal{C}^{\text T}\mathcal{C}+\gamma^{-2}\mathcal{Z}\mathcal{BB}^{\text T}\mathcal{Z}\prec 0.
\end{align*}
Pre and post multiplying the above LMI by the inverse of $\mathcal{Z}$, namely $\mathcal{Y}$, we have
\begin{align*}
(\mathcal{A}+\Delta_{\mathcal{A}})\mathcal{Y}+\mathcal{Y}(\mathcal{A}+\Delta_{\mathcal{A}})^{\text T}+\mathcal{YC}^{\text T}\mathcal{CY}+\gamma^{-2}\mathcal{BB}^{\text T}\prec 0,\\
\mathcal{AY}+\mathcal{YA}^{\text T}+\gamma^{-2}\mathcal{BB}^{\text T}+\mathcal{YC}^{\text T}\mathcal{CY}+\Delta_{\mathcal{A}} \mathcal{Y+Y}\Delta_{\mathcal{A}}^{\text T}\prec 0.
\end{align*}
Plugging $\Delta_{\mathcal{A}}=\mathcal{D} \Delta \mathcal{E}$ into the previous inequality, we get
\begin{align*}
\mathcal{AY}+\mathcal{YA}^{\text T}+\gamma^{-2}\mathcal{BB}^{\text T}+\mathcal{YC}^{\text T}\mathcal{CY}+(\mathcal{D}\Delta \mathcal{E})\mathcal{Y}+\mathcal{Y}(\mathcal{D}\Delta \mathcal{E})^{\text T}\prec 0.
\end{align*}
Having $\Delta^{\text T}\Delta \preceq \rho^2 I_j$, the above inequality is valid for all acceptable values of $\Delta$ if and only if there exists a positive $\bar{\varepsilon}>0$ for which the following LMI holds.
\begin{align*}
\mathcal{AY}+\mathcal{YA}^{\text T}+\gamma^{-2}\mathcal{BB}^{\text T}+\mathcal{YC}^{\text T}\mathcal{CY}+\bar{\varepsilon}\rho \mathcal{DD}^{\text T}+\bar{\varepsilon}^{-1}\rho \mathcal{Y}\mathcal{E}^{\text T}\mathcal{EY}\prec 0.
\end{align*}
Multiplying both sides of the above inequality by $\gamma$, we will have
\begin{align*}
\mathcal{AX}+\mathcal{XA}^{\text T}+\gamma^{-1}\mathcal{BB}^{\text T}+\gamma^{-1}\mathcal{XC}^{\text T}\mathcal{CX}+\varepsilon \rho \mathcal{DD}^{\text T}+\varepsilon^{-1} \rho \mathcal{X}\mathcal{E}^{\text T}\mathcal{EX}\prec 0,
\end{align*}
where $\mathcal{X}=\gamma \mathcal{Y}$ and $\varepsilon=\bar{\varepsilon}\gamma$. It can effortlessly be verified that the last LMI is the Schur complement of the negative definite constraint
\begin{align*}
\left[\begin{array}{cccc}
\mathcal{Y}_1+\mathcal{Y}_1^{\text T}+\varepsilon {\rho} \mathcal{DD}^{\text T} &*&*&*\\
\mathcal{B}^{\text T}& -\gamma I_m& *& *\\
\mathcal{(CX)}&0&-\gamma I_q&*\\
\sqrt{\rho}\mathcal{EX} & 0 & 0 &-\varepsilon I_j
\end{array}\right] \prec 0,
\end{align*}
where $\mathcal{Y}_1=\mathcal{XA}^{\text T}$. The rest of the proof is a mere application of Lemma \ref{lem:non_to_rank}; hence, omitted. 
\end{proof}
Consequently, we can reformulate the problem (\ref{eq:H2HiOSP}) into a rank-constrained problem, as described in the sequel.
\begin{theorem}
The optimization problem (\ref{eq:H2HiOSP}) is equivalent to the following rank-constrained optimization problem
\begin{alignat}{3} 
&\underset{K,\varepsilon_{y},\varepsilon_{\mathcal{S}}}{\textrm{\emph{minimize}}}~~&& \varepsilon_{\mathcal{S}}+ \lambda_1 \varepsilon_y+\lambda_2 \|K\|_{0} \hspace{85pt} &\label{eq:H2HiOSP_rank}\\
&\mbox{\textrm{\emph{subject to}}}~~&& K\in\mathcal{K},\notag\\
& &&X_r\succ 0,~~ r=1,2,\notag\\
& &&\varepsilon_r>0,~~~r=1,2,\notag\\
& &&\mathbf{Tr}(\bar{C}X_1\bar{C}^{\text T})<\varepsilon_{\mathcal{S}},\notag \\
& &&\left[\begin{array}{cc}
P_1+\bar{B}\bar{B}^{\text T} & *\\
\sqrt{\rho} Y_2^{\text T} & -\varepsilon_1 I_j \end{array}\right] \prec 0, \notag\\
& &&\left[\begin{array}{cccc}
P_2 &*&*&*\\
\bar{B}^{\text T}& -\varepsilon_y I_{p}& *& *\\
(\bar{C}X_2)&0&-\varepsilon_y I_{q}&*\\
\sqrt{\rho}Y_4^{\text T} & 0 & 0 &-\varepsilon_2 I_j
\end{array}\right] \prec 0,\notag \\
& &&\mathbf{rank}(M_1)=2n, \notag 
\end{alignat}
where
\begin{align*}
P_r&=Y_{2r-1}+Y_{2r-1}^{\text T}+\varepsilon_r {\rho} \bar{D}\bar{D}^{\text T},~~~~~~~~~~  r=1,2,\\
M_1&=\left[\begin{array}{ccccc} X_1 & *&  * & *&*\\
Y_1^{\text T} & - &* &*&*\\
Y_2^{\text T} & - & -& *&*\\
X_2& Y_3& Y_4&-&*\\
I_{2n} & A_{cl}^{\text T} & E_{cl}^{\text T}&-&-
\end{array} \right],\\
A_{cl}&=\mathbf{diag}(A+B_1KC,A+B_1\hat{K}C)\in\mathbb{R}^{2n \times 2n},\\
\bar{D}&=\left[\begin{array}{cc} D^{\text T} & 0\end{array}\right]^{\text T}\in\mathbb{R}^{2n \times i},\\
E_{cl}&=\left[\begin{array}{cc} E_A+E_{B_1}KC & 0\end{array}\right]\in\mathbb{R}^{j \times 2n}.
\end{align*}
\end{theorem}
\begin{proof}
It can be observed that the closed-loop system can be represented using the state representation $(A_{cl}+\bar{\Delta}_A,\bar{B},\bar{C},0)$, where $\bar{B}$ and $\bar{C}$ are defined in (\ref{eq:mat_bar}) and  $\bar{\Delta}_A=\bar{D}\Delta E_{cl}$. Therefore, applying the results from lemmas \ref{thm:H2_to_rank_uncert} and \ref{thm:Hi_to_rank_uncert} yields the desired result.
\end{proof}
\noindent The next corollary is now immediate.
\begin{corollary}
The optimization problem (\ref{eq:H2HiOSP}) can equivalently be cast as the following rank-constrained optimization problem
\begin{subequations}\label{eq:H2HiOSP_rank_cor}
\begin{alignat}{3}
&\underset{K,\varepsilon_{y},\varepsilon_{\mathcal{S}}}{\textrm{\emph{minimize}}}~~&& \varepsilon_{\mathcal{S}}+ \lambda_1 \varepsilon_y+\lambda_2 \|K\|_{0} \hspace{81pt}&\\
&\mbox{\textrm{\emph{subject to}}}~~&& K\in\mathcal{K}&\label{eq:H2HiOSP_rank_cor_first}\\
& &&X_r\succ 0,~~~ r=1,2,\\
& &&\varepsilon_r>0,~~~~ r=1,2,\\
& &&\mathbf{Tr}(\bar{C}X_1\bar{C}^{\text T})<\varepsilon_{\mathcal{S}}, \\
& &&\left[\begin{array}{c@{\hspace{.2em}}c}
 Q_1+\bar{B}\bar{B}^{\text T}+\varepsilon_1 {\rho} \bar{D}\bar{D}^{\text T} & *\\
 \sqrt{\rho} R_1 & -\varepsilon_1 I_j \end{array}\right] \prec 0, \\
& && \left[\begin{array}{c@{\hspace{.2em}}c@{\hspace{.2em}}c@{\hspace{.2em}}c}
  Q_2+\varepsilon_2 {\rho} \bar{D}\bar{D}^{\text T} &*&*&*\\
  \bar{B}^{\text T}& -\varepsilon_y I_{p}& *& *\\
  (\bar{C}X_2)&0&-\varepsilon_y I_{q}&*\\
  \sqrt{\rho}R_2 & 0 & 0 &-\varepsilon_2 I_j
  \end{array}\right] \prec 0,\label{eq:H2HiOSP_rank_cor_last}\\
& &&\mathbf{rank}(M_2)=2n, \label{eq:H2HiOSP_rank_cor_rank}
\end{alignat}
\end{subequations}
where
\begin{subequations}
\begin{alignat}{1}
Q_r&=X_rA_o^{\text T}+A_oX_r+Y_rB_K^{\text T}+B_K^{\text T}Y_r^{\text T},~~r=1,2\label{eq:H2HiOSP_rank_cor_par_first}\\
R_r&=E_{o}X_r+E_{B_1} Y_{r}^{\text T},~~~~~~~~~~~~~~~~~~~~~~~r=1,2\\
M_2&=\left[\begin{array}{cccc} 
X_1 & *&  * & *\\
Y_1^{\text T} & - &* &*\\
X_2 & Y_2& -& *\\
I_{2n} & (KC_K)^{\text T}&-&-
\end{array} \right],  \label{eq:M-2}\\
A_{o}&=\mathbf{diag}(A,A+B_1\hat{K}C)\in\mathbb{R}^{2n \times 2n},\\
\bar{D}&=\left[\begin{array}{cc} D^{\text T} & 0\end{array}\right]^{\text T}\in\mathbb{R}^{2n \times i},\\
E_{o}&=\left[\begin{array}{cc} E_A & 0\end{array}\right]\in\mathbb{R}^{j \times 2n},\\
C_K&=\left[\begin{array}{cc} C & 0\end{array}\right]\in\mathbb{R}^{{q} \times 2n},\\
B_K&=\left[\begin{array}{cc} B_1^{\text T} & 0\end{array}\right]^{\text T}\in\mathbb{R}^{2n \times m}.\label{eq:H2HiOSP_rank_cor_par_last}
\end{alignat}
\end{subequations}
\end{corollary}
\vspace{-6pt}
\section{A Tractable Approximation Algorithm for Computing Sparse Feedback Controllers} \label{sec:algorithm} 
\vspace{-2pt}
Although both optimizations (\ref{eq:H2HiOSP_rank}) and (\ref{eq:H2HiOSP_rank_cor}) can be utilized to solve our controller sparsification problem, we choose to only implement the one formulated in (\ref{eq:H2HiOSP_rank_cor}). The terms in our optimization problem are all convex except the sparsity-promoting term in the cost function and the rank constraint. This section intends to shed light on our approach in dealing with these two non-convex and combinatorial terms. 

As for the sparsity-promoting term of the objective function, since the $\ell_0$ norm is an integer-valued function, utilizing it in our formulation introduces the complications of combinatorial optimization. In order to reduce the complexity of sparse vector/matrix recovery problems, we employ the $\ell_1$ norm and its weighted versions. This is because convex surrogates of the $\ell_0$ norm are among the most common functions used to measure the sparsity and have been utilized in diverse applications \cite{Lin:2013,Arastoo:2012}. Therefore, we have
\begin{alignat}{3}
&\underset{K,\varepsilon_{y},\varepsilon_{\mathcal{S}}}{\textrm{{minimize}}}~~&& \varepsilon_{\mathcal{S}}+ \lambda_1 \varepsilon_y+\lambda_2 \|W\circ K\|_{1}\hspace{63pt}& \label{eq:H2HiOSP_rank_cor_rel}\\
&\mbox{\textrm{{subject to}}}~~&&(\ref{eq:H2HiOSP_rank_cor_first})-(\ref{eq:H2HiOSP_rank_cor_rank}),\notag\\
& &&(\ref{eq:H2HiOSP_rank_cor_par_first})-(\ref{eq:H2HiOSP_rank_cor_par_last}),\notag
\end{alignat}
where the weight matrix $W=\left[w_{ij}\right]\in\mathbb{R}^{m \times q}$ is entry-wise positive and chosen according to the objectives of the problem.

The convex relaxation of the sparsity-promoting term in the cost function of (\ref{eq:H2HiOSP_rank_cor_rel}) leaves us with an optimization problem in which non-convexity only arises in the form of a rank constraint, i.e., $\mathbf{rank}(M_2)=2n$. It is known that presence of the rank constraint still causes our optimization problem to become NP-hard. Therefore, we propose a technique, which is built upon the method studied in \cite{Arastoo:2016ACC}, to solve the rank constraint optimization problem. In a nutshell, this method is based on substituting the rank constraint on the symmetric matrix $M_2$ with a positive semidefinite constraint while introducing extra convex constraints along with a bi-linear term to the cost function. 
\begin{theorem}
Let us consider the rank-constrained optimization problem (\ref{eq:H2HiOSP_rank_cor_rel}) and define the following auxiliary optimization problem 
\begin{align}
\underset{Y,K,\varepsilon_{y},\varepsilon_{\mathcal{S}}}{\textrm{\emph{minimize}}}~~& \varepsilon_{\mathcal{S}}+ \lambda_1 \varepsilon_y+\lambda_2 \|W\circ K\|_{1} + \nu \mathbf{Tr}(YM_2) \label{eq:H2HiOSP_cor_bilin}\\
 \mbox{\textrm{\emph{subject to}}}~~&(\ref{eq:H2HiOSP_rank_cor_first})-(\ref{eq:H2HiOSP_rank_cor_last}),\notag\\
&(\ref{eq:H2HiOSP_rank_cor_par_first})-(\ref{eq:H2HiOSP_rank_cor_par_last}),\notag\\
&0\preceq Y\preceq I_{6n+m},\notag\\
&\mathbf{Tr}(Y)=4n+m,\notag\\
&M_2 \succeq 0,\notag
\end{align}
in which $\lambda_1, \lambda_2, \nu >0$ and the entry-wise positive matrix $W$ are some given design parameters. If  problem (\ref{eq:H2HiOSP_rank_cor_rel}) is feasible, 
then there exists a constant $\eta > 0$ for which the optimal solution $M_2$ from solving \eqref{eq:H2HiOSP_cor_bilin} satisfies  
\begin{align*}
    \mathbf{rank}(M_2;\eta \nu^{-1})~\leq~ 2n,
\end{align*}
i.e., rank of $M_2$ is less than or equal to $2n$ with tolerance threshold $\eta \nu^{-1}$ according to  Definition \ref{rank-tol}.
\end{theorem}
\begin{proof}
Examining the  optimization problem
\begin{align}
\underset{Y,K,\varepsilon_{y},\varepsilon_{\mathcal{S}}}{\textrm{{minimize}}}~~& \varepsilon_{\mathcal{S}}+ \lambda_1 \varepsilon_y+\lambda_2 \|W\circ K\|_{1} + \nu \mathbf{Tr}(YM_2)\label{eq:H2HiOSP_cor_bilin_rank}\\
\mbox{\textrm{{subject to}}}~~&(\ref{eq:H2HiOSP_rank_cor_first})-(\ref{eq:H2HiOSP_rank_cor_last}),\notag\\
&(\ref{eq:H2HiOSP_rank_cor_par_first})-(\ref{eq:H2HiOSP_rank_cor_par_last}),\notag\\
&0\preceq Y\preceq I_{6n+m},\notag\\
&\mathbf{Tr}(Y)=4n+m, \notag
\end{align}
and considering the assumption of feasibility of (\ref{eq:H2HiOSP_rank_cor_rel}), it is straightforward to show that the optimal cost of this problem is equal to the optimal cost of (\ref{eq:H2HiOSP_rank_cor_rel}) for all values of $\nu$. This is due to the fact that the optimal value of $\mathbf{Tr}(YM_2)$, which can be shown to be the sum of the $4n+m$ smaller singular values of matrix $M_2$ \cite[p.266]{Dattoro:2005}, is always equal to zero, since $\mathbf{rank}(M_2)=2n$. 

Also, it is easy to establish that the feasible set of the optimization problem (\ref{eq:H2HiOSP_cor_bilin}) is a super-set of that of the (\ref{eq:H2HiOSP_cor_bilin_rank}). Hence, the optimal cost of (\ref{eq:H2HiOSP_cor_bilin}) is bounded above by the optimal cost of (\ref{eq:H2HiOSP_cor_bilin_rank}). Therefore, denoting the optimal cost of (\ref{eq:H2HiOSP_rank_cor_rel}) by $J^*$, we can write
\begin{align*}
J^*&\geq \varepsilon^*_{\mathcal{S}}(\nu)+ \lambda_1 \varepsilon^*_y(\nu)+\lambda_2 \|W\circ K^*(\nu)\|_{1} \\
&+ \nu \mathbf{Tr}(Y^*(\nu)M^*_2(\nu)),
\end{align*}
where $\varepsilon^*_{\mathcal{S}}(\nu)$, $\varepsilon^*_y(\nu)$, $K^*(\nu)$, $Y^*(\nu)$, and $M^*_2(\nu)$ are the optimal solutions to (\ref{eq:H2HiOSP_cor_bilin}) for that particular value of $\nu$. Since the first three terms on the right hand side of the above inequality are non-negative, we can deduce
\begin{align*}
    \mathbf{Tr}(Y^*(\nu)M^*_2(\nu))\leq J^* \nu^{-1}.
\end{align*}
Therefore, the sum of $4n+m$ smaller singular values of matrix $M_2$ is less or equal to $J^* \nu^{-1}$. Taking the constant $\eta$ greater than or equal to $J^*$, we have
\begin{align*}
    \mathbf{rank}(M_2;\eta \nu^{-1})\leq 2n.
\end{align*}
Hence, the proof is complete.
\end{proof}
%
\begin{figure}[!t]
\vspace{.1in}
\begin{center}
\begin{tabular}{| l |}
\hline
{\bf Algorithm 1:} Solution to problem (\ref{eq:H2HiOSP_cor_bilin}) \\
\hline 
\hline
{\bf Inputs:} $A$, $B_1$, $B_2$, $C$, $Q$, $R$, $\lambda_1$, $\lambda_2$, $\nu$, $\mathcal{K}$, $W$, $\rho$, and $\varepsilon^*$.\\
~1: {\em Initialization:} \\
~~~~{\small Set $Y^{(0)}=I_{6n+m}$, $\varepsilon^{(0)} > \varepsilon^*$, $K^{(0)}=0_{m \times q}$ and $k=0$.}\\
~2: {\bf While} $\varepsilon^{(k)} > \varepsilon^*$  \bf{do}\\
~3: \hspace{.1in} Update $Z^{(k+1)}$ by solving (\ref{eq:Z_min}),\\
~4: \hspace{.1in} Update $Y^{(k+1)}$ using the equation (\ref{eq:Y_min}),\\
~5: \hspace{.1in} Update $\varepsilon^{(k+1)}$ using the equation (\ref{eq:stop_crit}),\\
~6: \hspace{.1in} $k\leftarrow k+1$,\\
~7: {\bf end while}\\
~8: Truncate $K$.\\
{\bf Output:} $K$ \\
\hline
\end{tabular}
\normalsize
\end{center}
\vspace{-.25in}
\end{figure}
%
We should remind that according to (\ref{eq:M-2}) and the specific structure of matrix $M_2$ it is always true that $\mathbf{rank}(M_2) \geq 2n$. As a result of the previous theorem, we can now solve the optimization problem (\ref{eq:H2HiOSP_cor_bilin}) for an appropriately-chosen parameter $\nu$ to obtain a sub-optimal solution to the problem (\ref{eq:H2HiOSP_rank_cor_rel}). For the simplicity of our notations, the letter $Z$ is used to denote the stack of all optimization variables excluding variable $Y$. The optimization problem (\ref{eq:H2HiOSP_cor_bilin}) can be rewritten as follows.
\begin{alignat*}{3}
&\underset{Z,Y}{\textrm{{minimize}}}~~&& \mathcal{F}(Z,Y)\hspace{155pt} &\\
&\mbox{\textrm{{subject to}}}~~&& Z\in \mathcal{C}_z,~~Y\in \mathcal{C}_y,
\end{alignat*}
where $\mathcal{C}_z$ is the convex set defined by the constraints (\ref{eq:H2HiOSP_rank_cor_first})-(\ref{eq:H2HiOSP_rank_cor_last}), (\ref{eq:H2HiOSP_rank_cor_par_first})-(\ref{eq:H2HiOSP_rank_cor_par_last}), along with $M_2\succeq0$, and the convex set $\mathcal{C}_y$ is generated by $\mathbf{Tr}(Y)=4n+m$ and 
$0\preceq Y\preceq I_{6n+m}$. Needless to say that $\mathcal{F}(Z,Y)$ represents the bi-linear objective function in the minimization problem (\ref{eq:H2HiOSP_cor_bilin}). The above reformulation allows us to carry out this problem by iteratively optimizing the objective function for $Z$ and $Y$. As a result, the main steps of this iterative method can be divided into two sub-problems, $Z$-minimization and $Y$-minimization problems.  As both $Z$-minimization and $Y$-minimization steps are convex optimizations, they can be performed in a computationally efficient manner. However, for the $Y$-minimization, there also exists an analytic solution, stated in the next theorem.

\begin{theorem} \label{thm:Y_min_sol}
The optimal solution to the \emph{$Y$-minimization} step is given by
\begin{align}\label{eq:Y_update}
    Y^*=I_{6n+m}-\sum_{i=1}^{2n} u_i {u_i}^{\text T},
\end{align}
where vectors $u_i$ for $i=1,\ldots,2n$ are the singular vectors corresponding to the $2n$ larger singular values of $M_2$.
\end{theorem}
\begin{proof}
The optimal value of $\mathbf{Tr}(YM_2)$, subject to $0\preceq Y\preceq I_{6n+m}$ and $Y \succeq 0$, is the sum of the $4n+m$ smaller singular values of matrix $M_2$ \cite[p.266]{Dattoro:2005}. The rest of the proof is straightforward.
\end{proof}
\subsection{Summary of the Algorithm} We utilize the following sequence of iterations to obtain the minimizer of the constrained problem (\ref{eq:H2HiOSP_cor_bilin}). First, we solve the $Z$-minimization and $Y$-minimization subproblems
\begin{align}
Z^{(k+1)}&=\arg\underset{Z\in \mathcal{C}_z} {\textrm{minimize}}~~\mathcal{F}(Z,Y^{(k)}), \label{eq:Z_min}\\
Y^{(k+1)}&=I_{6n+m}-\sum_{i=1}^{2n} u_i^{(k+1)} {u_i^{(k+1)}}^{\text T} ,\label{eq:Y_min}
\end{align}
where $M_2^{(k+1)}=\sum_{i=1}^{6n+m} \sigma_i^{(k+1)} u_i^{(k+1)} {u_i^{(k+1)}}^{\text T}$ is the singular value decomposition of $M_2^{(k+1)}$. The stopping criterion is established by $\varepsilon^{(k+1)}\le \varepsilon^*$, where  $\varepsilon^*$ is the given desired precision, with the following update law
\begin{equation} \label{eq:stop_crit}
\varepsilon^{(k+1)}=\frac{\|K^{(k+1)}-K^{(k)}\|}{\|K^{(k+1)}\|}.
\end{equation}
In the last step of the algorithm, we truncate negligible entries of the resulting feedback gain $K$,  e.g., those smaller than $5 \times 10^{-5}$,. These small entries show very weak couplings between the nodes in the information structure of the controller. A summary of our proposed algorithm is described in Algorithm 1.
\begin{remark}\label{rem:mon-dec}
Let $\left\{(Z^{(k)},Y^{(k)})\right\}$ be the sequence generated by Algorithm 1. Then, $\left\{\mathcal{F}(Z^{(k)},Y^{(k)})\right\}$ is a monotonically decreasing sequence and, hence, convergent.
\end{remark}

%

\begin{remark}
The choice of the weight matrix $W$ plays an important role in the sparsity-promoting properties of our method. When a proper weight matrix is not accessible, the weighted $\ell_1$ norm technique can also be employed to enhance the sparse controller recovery. In this method, the weight assigned to each controller entry is updated inversely proportional to the value of the corresponding matrix entry recovered from the previous iteration, i.~e.
\begin{align}
&w_{ij}^{(k+1)}=1/(|k_{ij}^{(k)}|+\xi), \:\:\: \forall i,j, \label{eq:w_update}
\end{align}
where the constant $\xi>0$ which is chosen as a relatively small constant, is augmented to the denominator of the update law (\ref{eq:w_update}) to guarantee the stability of the algorithm, especially, when $k_{ij}^{(k)}$ turns out to be zero in the previous iteration \cite{Candes:2008}. It should be noted that, our simulation results are obtained by performing this update law only in the first few iterations to avoid adversely affecting the convergence of the algorithm.
\end{remark}
\begin{figure}[!t]
\centering
  {\includegraphics[trim = 25mm 33mm 25mm 0mm, clip,width=0.45\textwidth]{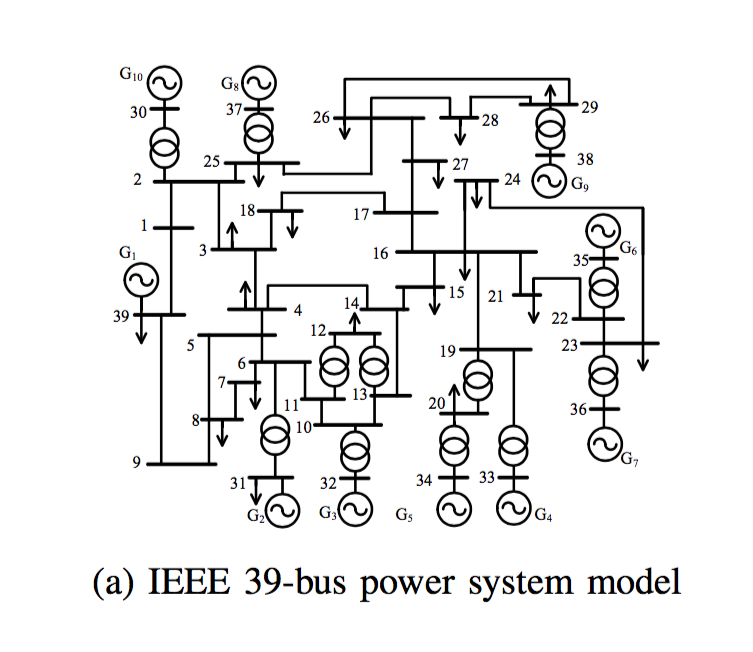}}
      \caption{IEEE 39-Bus Power System Model}
    \label{PowNet}
\end{figure}
\vspace{-6pt}
\section{Simulation Results}\label{sec:simul}
\vspace{-2pt}
In this section, we examine our proposed method by utilizing the IEEE $39$-Bus New England power system which consists of $N_G=10$ synchronous generators. Specifically, we take advantage of the state-space model provided by \cite{Kalbat:2014}, which is a linearized state-space model of swing equations is characterized by (\ref{eq:system}), where
\begin{align*}
x&= \left[\begin{array}{cc} \theta^{\text T} & \omega^{\text T} \end{array}\right]^{\text T},\\
A&=\left[\begin{array}{cc} 0 & I \\ -{\tilde{M}}^{-1}L & -{\tilde{M}}^{-1}\tilde{D} \end{array}\right],\\
B_1&=\left[\begin{array}{c} 0 \\ {\tilde{M}}^{-1} \end{array}\right] ,\\
B_2&=B_1\\
\tilde{M}& = \mathbf{diag}(\tilde{M}_1,\cdots,\tilde{M}_{N_G}),\\
\tilde{D}&= \mathbf{diag}(\tilde{D}_1,\cdots,\tilde{D}_{N_G}),\\
u&=Kx,\\
K&=\left[\begin{array}{cc}K_{\theta} & K_{\omega} \end{array}\right],\\
    \omega&=\dot{\theta}.
\end{align*}
The Laplacian or admittance matrix $L$ satisfies the following equations.
\begin{align*}
l_{ij}&= -b_{ij}^{\text{\it Kron}},\\
l_{ii}&=\sum_{k=1,k \neq i}^{N_G} b_{ik}^{\text{\it Kron}},
\end{align*}
where $B^{\text{\it Kron}}$ is the susceptance matrix of the corresponding Kron reduced admittance matrix. 

The power network utilized in our simulation is depicted in Figure \ref{PowNet}, and its parameters, in per unit system, are presented in Table \ref{table:1}. 

\begin{table}[!b]
\small
\centering
\begin{tabular}{|c| c| c |c|c|c|} 
 \hline
Bus & Generator & $\tilde{M}_i$ & $\tilde{D}_i$ & $\theta_0$ & $\omega_0$ \\ 
\hline
$30$ & $G_{10}$ & $4$ & $5$ & $-0.0839$ & $1$\\ 
\hline
$31$ & $G_{2}$ & $3$ & $4$ & $0.0000$ & $1$\\
 \hline
 $32$ & $G_{3}$ & $2.5$ & $4$ & $0.0325$ & $1$\\
 \hline
 $33$ & $G_{4}$ & $4$ & $6$ & $0.0451$ & $1$\\
 \hline
 $34$ & $G_{5}$ & $2$ & $3.5$ & $0.0194$ & $1$\\
 \hline
 $35$ & $G_{6}$ & $3.5$ & $3$ & $-0.0073$ & $1$\\
 \hline
 $36$ & $G_{7}$ & $3$ & $7.5$ & $0.1304$ & $1$\\
 \hline
 $37$ & $G_{8}$ & $2.5$ & $4$ & $0.0211$ & $1$\\
 \hline
 $38$ & $G_{9}$ & $2$ & $6.5$ & $0.1270$ & $1$\\
 \hline
 $39$ & $G_{1}$ & $6$ & $5$ & $-0.2074$ & $1$\\
 \hline
\end{tabular}
\vspace{.25in}
\caption{Power parameters used in our simulations.}
\label{table:1}
\normalsize
\end{table}

\begin{figure}[t]
\centering
\vspace{-.2cm}
  \hspace{10mm}  \subfloat[\label{fig:rho0}]{%
      \includegraphics[trim = 20mm 25mm 10mm 28mm, clip,width=0.4\textwidth]{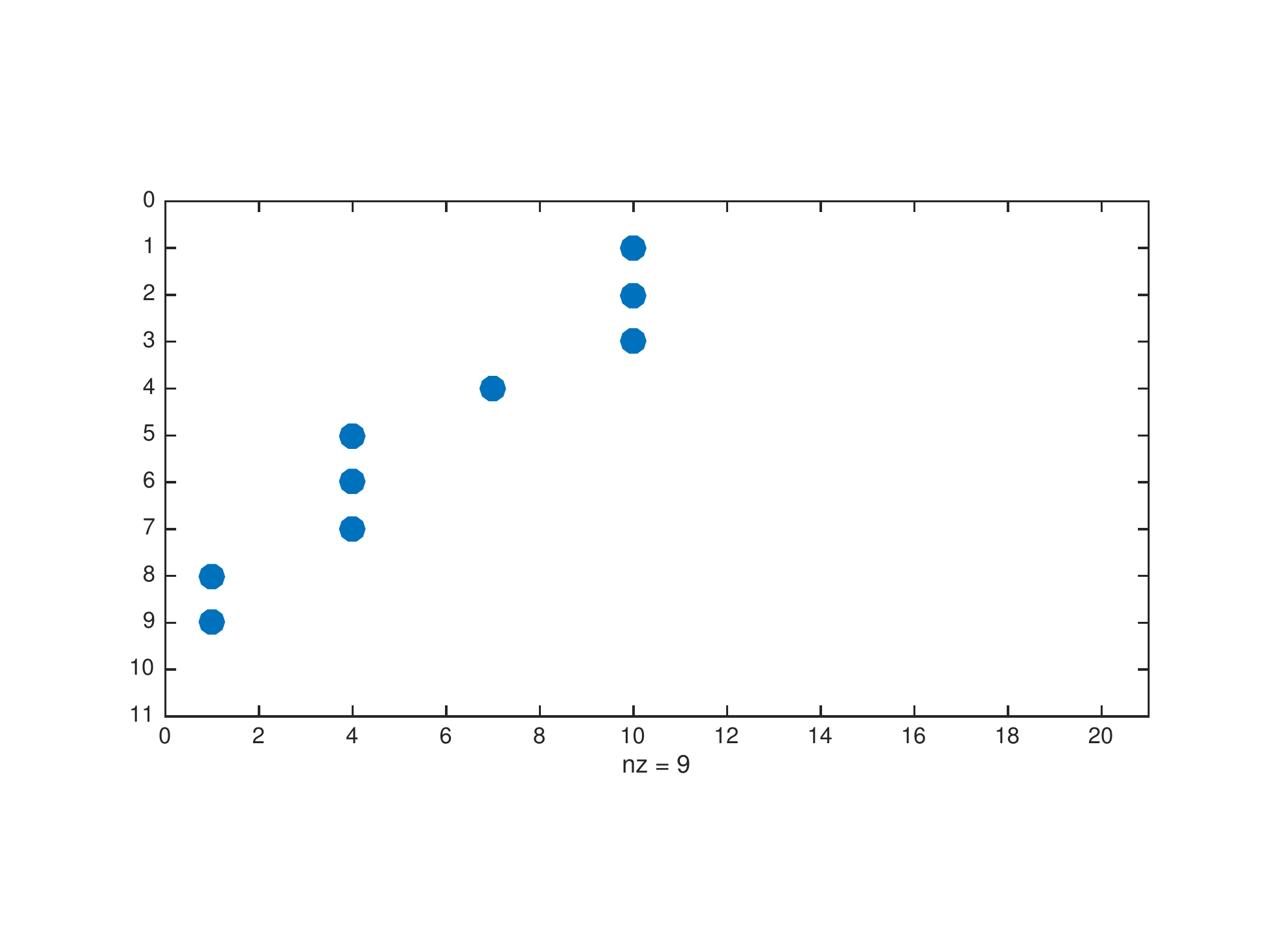}
    }
        \hfill
    \subfloat[\label{fig:rhoN0}]{%
      \includegraphics[trim = 20mm 25mm 10mm 28mm, clip, width=0.4\textwidth]{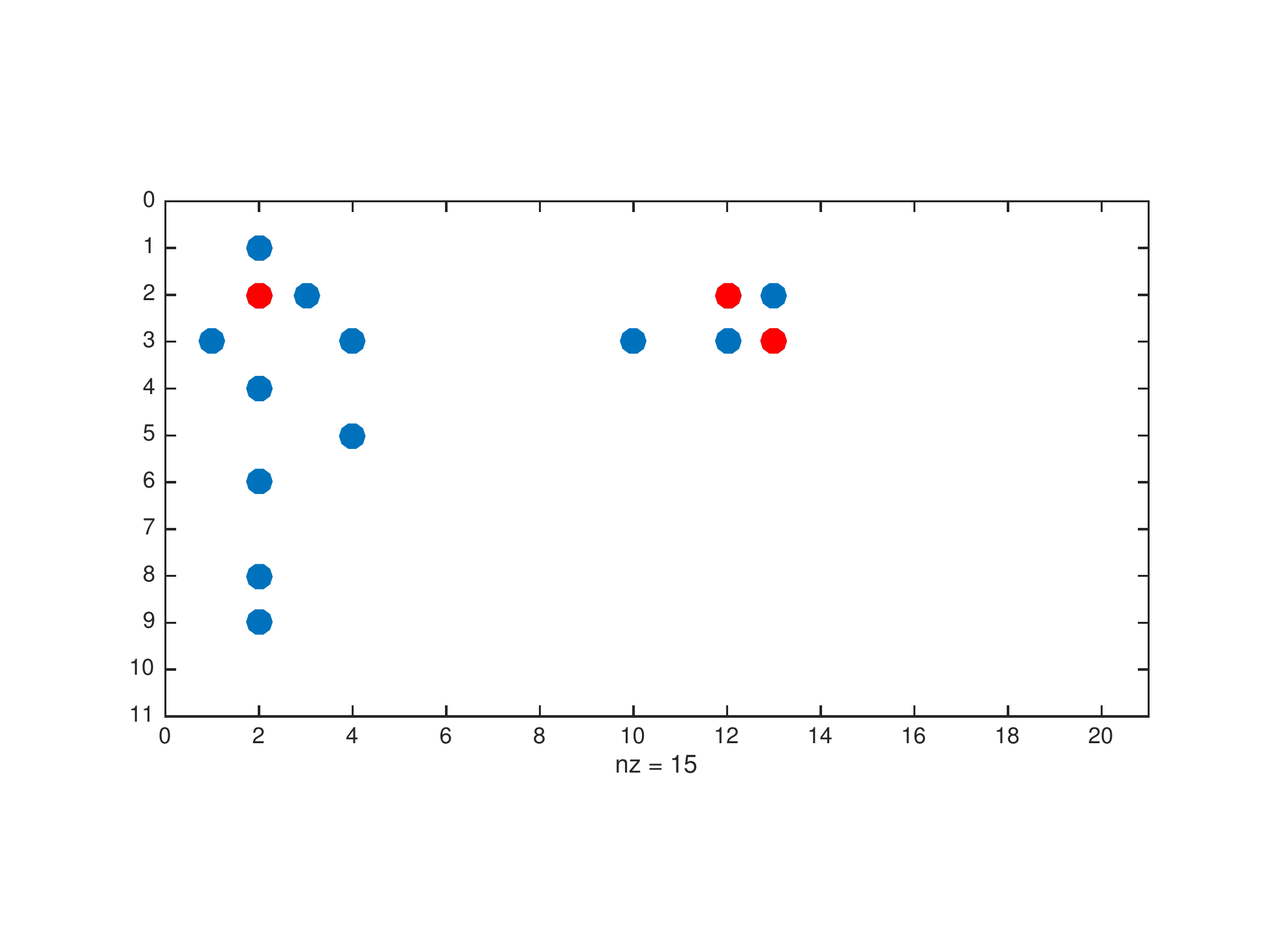}
      }    \hspace{10mm}
      
   \hspace{11mm}   \subfloat[\label{fig:rhoG0}]{%
      \includegraphics[trim = 39mm 15mm 28mm 10mm, clip, width=0.385\textwidth]{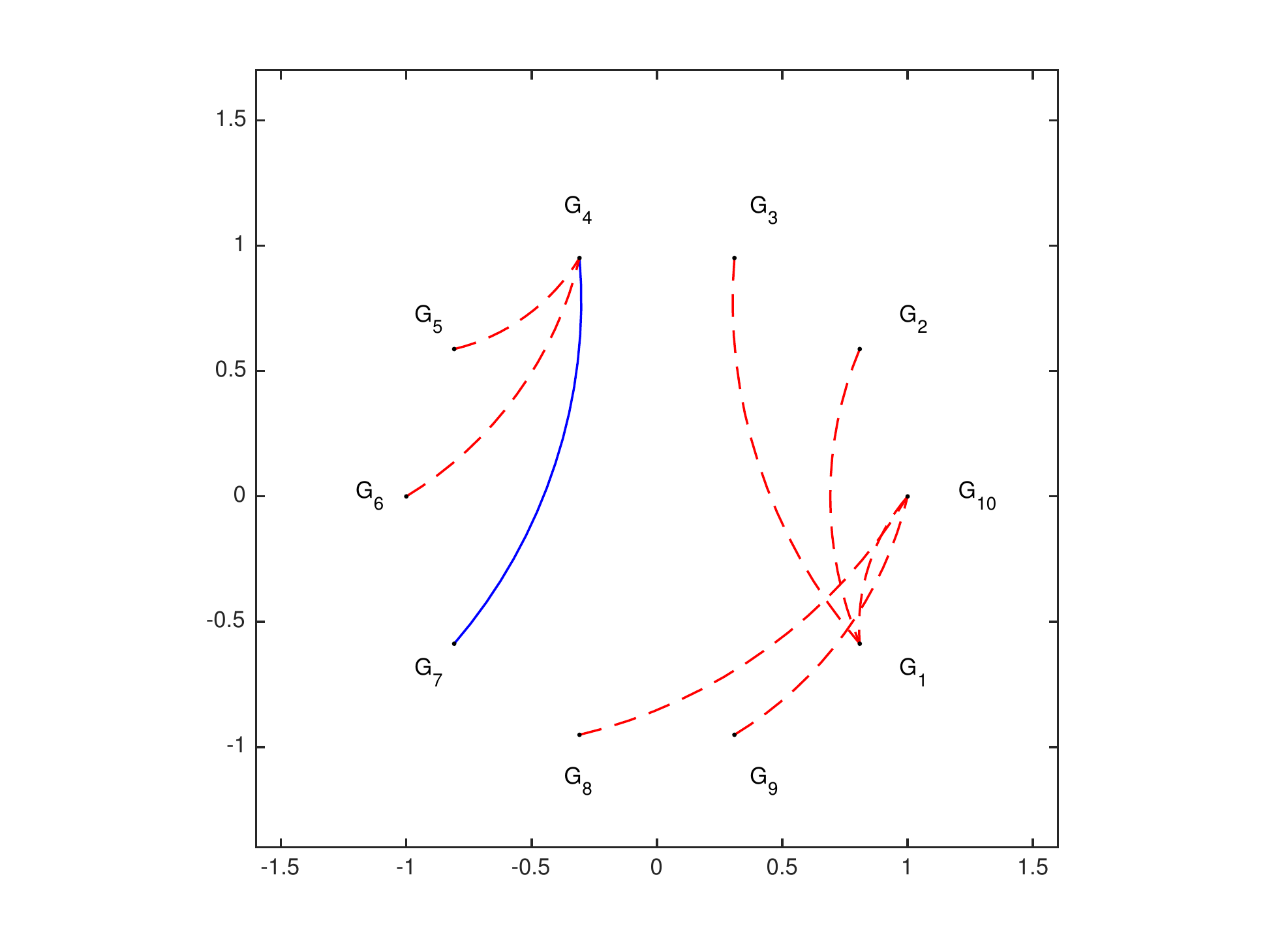}
      }
    \hfill
 \subfloat[\label{fig:rhoGN0}]{%
      \includegraphics[trim = 39mm 15mm 28mm 10mm, clip, width=0.385\textwidth]{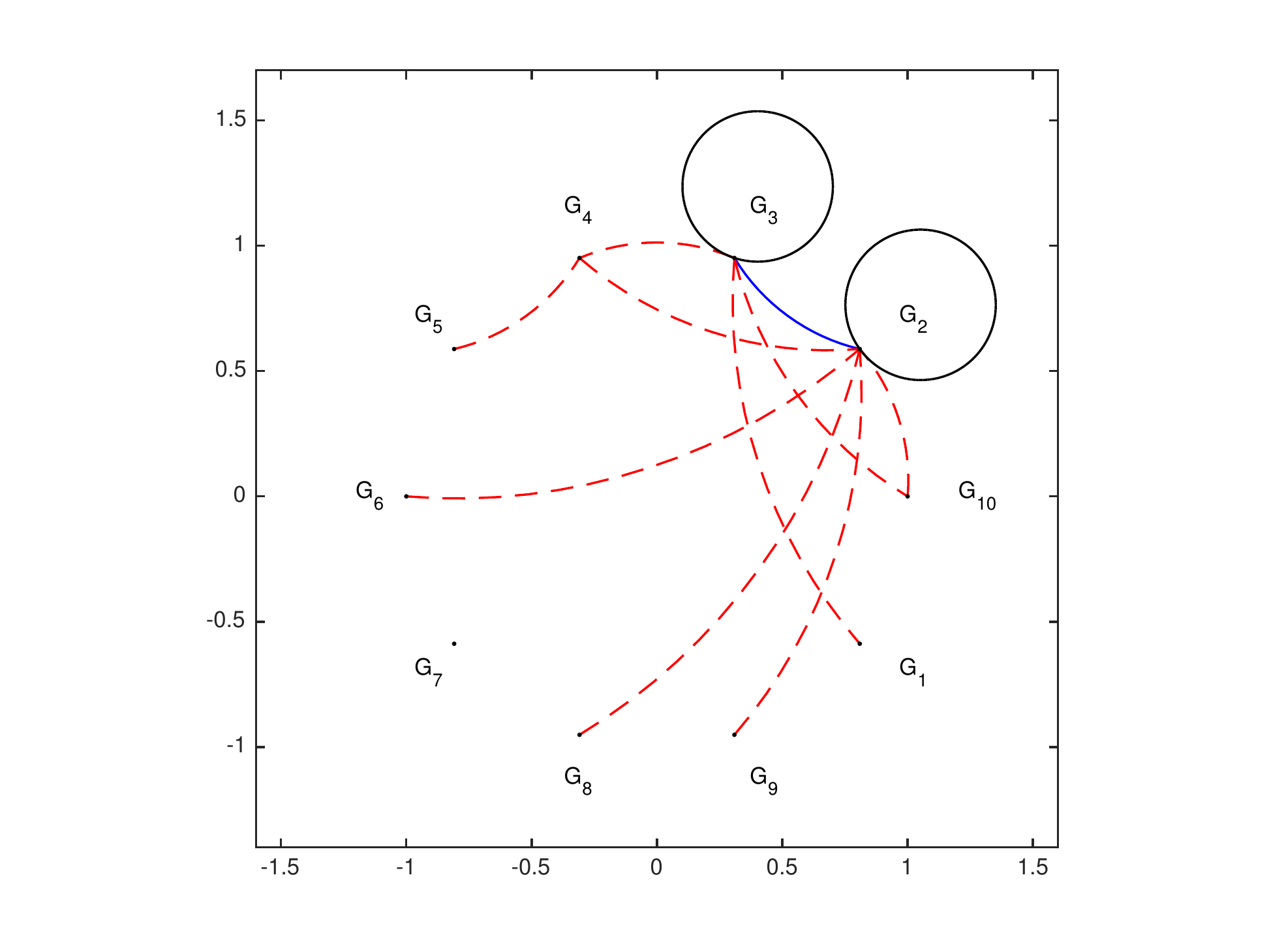}
      }   \hspace{11mm}
      \caption{ (a) Sparsity Pattern of $K$ for $\rho_{\text{\it rel}}=0 \%$; Blue and red bullets are used to depict diagonal and off-diagonal entries of $K$, respectively (b) Sparsity Pattern of $K$ for $\rho_{\text{\it rel}}= 30 \%$ \big($(i_1,i_2)=(2,3)$\big) (c) Sparsity Graph of $|K_{\theta}|+|K_{\omega}|$ for $\rho_{\text{\it rel}}=0 \%$; Blue solid lines, red dashed lines, and black self-loops are used to depict doubly-connected, singly-connected, and self-connected edges of $|K_{\theta}|+|K_{\omega}|$, respectively (d) Sparsity Graph of $|K_{\theta}|+|K_{\omega}|$ for $\rho_{\text{\it rel}}= 30 \%$ \big($(i_1,i_2)=(2,3)$\big) .}
    \label{fig:vs_lambda}
\end{figure}
We define the following performance metrics which quantify the deviation in $\mathcal{H}_2$ and $\mathcal{H}_\infty$ norms casued by the sparsification process. They also allow for comparison of the sparsification performance in the absence and presence of uncertainty on the system matrices.
\begin{align}
\mathcal{R}_2  =\frac{\|\mathcal{S}-\hat{\mathcal{S}}\|_{\mathcal{H}_2}}{\|\hat{\mathcal{S}}\|_{\mathcal{H}_2}},\\
\mathcal{R}_{\infty} =\frac{\|\mathcal{S}-\hat{\mathcal{S}}\|_{\mathcal{H}_{\infty}}}{\|\hat{\mathcal{S}}\|_{\mathcal{H}_{\infty}}}.
\end{align}
Now, we assume that the susceptance corresponding to the link between two randomly-chosen nodes $i_1$ and $i_2$ is affected by an uncertainty of the form
\begin{equation*}
    \rho = \rho_{\text{\it rel}} (b_{i_1i_2}^{\text{\it Kron}}),
\end{equation*}
where $\rho_{\text{\it rel}}$ is called the relative uncertainty and  $b_{ij}^{\text{\it Kron}}$ is assumed to take non-zero values. Should we relax this assumption, the uncertainty and relative uncertainty will have to be defined in a different way, e.g. 
\begin{align*}
\rho = \rho_{\text{\it rel}} \min \bigg\{\sum_{k=1,k \neq i_1}^{N_G} b_{i_1k}^{\text{\it Kron}},\sum_{k=1,k \neq i_2}^{N_G} b_{i_2k}^{\text{\it Kron}} \bigg\}.
\end{align*}
To study the effect of adding uncertainty to the link between generators $i_1$ and $i_2$, the matrices $D$, $E_A$, and $E_{B_1}$ are chosen as follows.
\begin{align*}
&D = -\begin{bmatrix} 0 & 0 \\ 0 & {\tilde{M}}^{-1} \end{bmatrix} (e_{i_1+N_G} -e_{i_2+N_G}),\\
&E_A =e_{i_1+N_G}^{\text T} -e_{i_2+N_G}^{\text T},\\
&E_{B_1}=0.
\end{align*}
Assuming $C=I$, $Q=I$, $R=10I$, $\lambda_1=0.5$, $\lambda_2=0.1$, $\nu=100$, $\xi=10^{-6}$, and $\varepsilon^*=10^{-2}$, we randomly choose two generators, $i_1=2$ and $i_2=3$, and consider the uncertainty cases $\rho_{\text{\it rel}} \in \{0 \%, 30 \%\}$. The results of the static state feedback controller design using our method are presented in Table \ref{table:2}. According to this table, an increase in uncertainty increases $\mathcal{R}_2$ and $\mathcal{R}_\infty$, and worsens the sparsification of the controller. 
\begin{table}[!b]
\small
\centering
\begin{tabular}{|c|c|c|c|} 
 \hline
$\rho_{\text{\it rel}}$ & $\mathcal{R}_2$ & $\mathcal{R}_{\infty}$ & $\|K\|_0/\|\hat{K}\|_0$  \\ 
 \hline
$0 \%$ & $21.35 \%$ & $49.42 \%$ & $4.5 \%$ \\ 
 \hline
$30 \%$ & $36.31 \%$ & $88.71 \%$ & $7.5 \%$ \\
 \hline
\end{tabular}
\caption {Performance and cardinality quantities for the case $\rho_{\text{\it rel}} \in \{0 \%, 30 \%\}$.}
\normalsize
\label{table:2}
\end{table}

Figures \ref{fig:rho0} and \ref{fig:rhoN0} visualize the corresponding sparsity patterns for both cases $\rho_{\text{\it rel}}= 0 \%$ and $\rho_{\text{\it rel}}= 30 \%$, respectively, and Figures \ref{fig:rhoG0} and \ref{fig:rhoGN0} visualize the corresponding sparsity graphs for both cases $\rho_{\text{\it rel}}= 0 \%$ and $\rho_{\text{\it rel}}= 30 \%$, respectively. It should be noted that entries $k_{22}$, $k_{23}$, $k_{2(12)}$, $k_{2(13)}$, $k_{3(12)}$, and $k_{3(13)}$ take non-zero values after applying the $30 \%$ relative uncertainty. The interpretation is that, since uncertainty causes interference to the link between two randomly-chosen generators, the construction of communication links between such generators is vital.

\begin{figure}[!t]
\centering
  {\includegraphics[trim = 38mm 10mm 28mm 10mm, clip,width=0.4\textwidth]{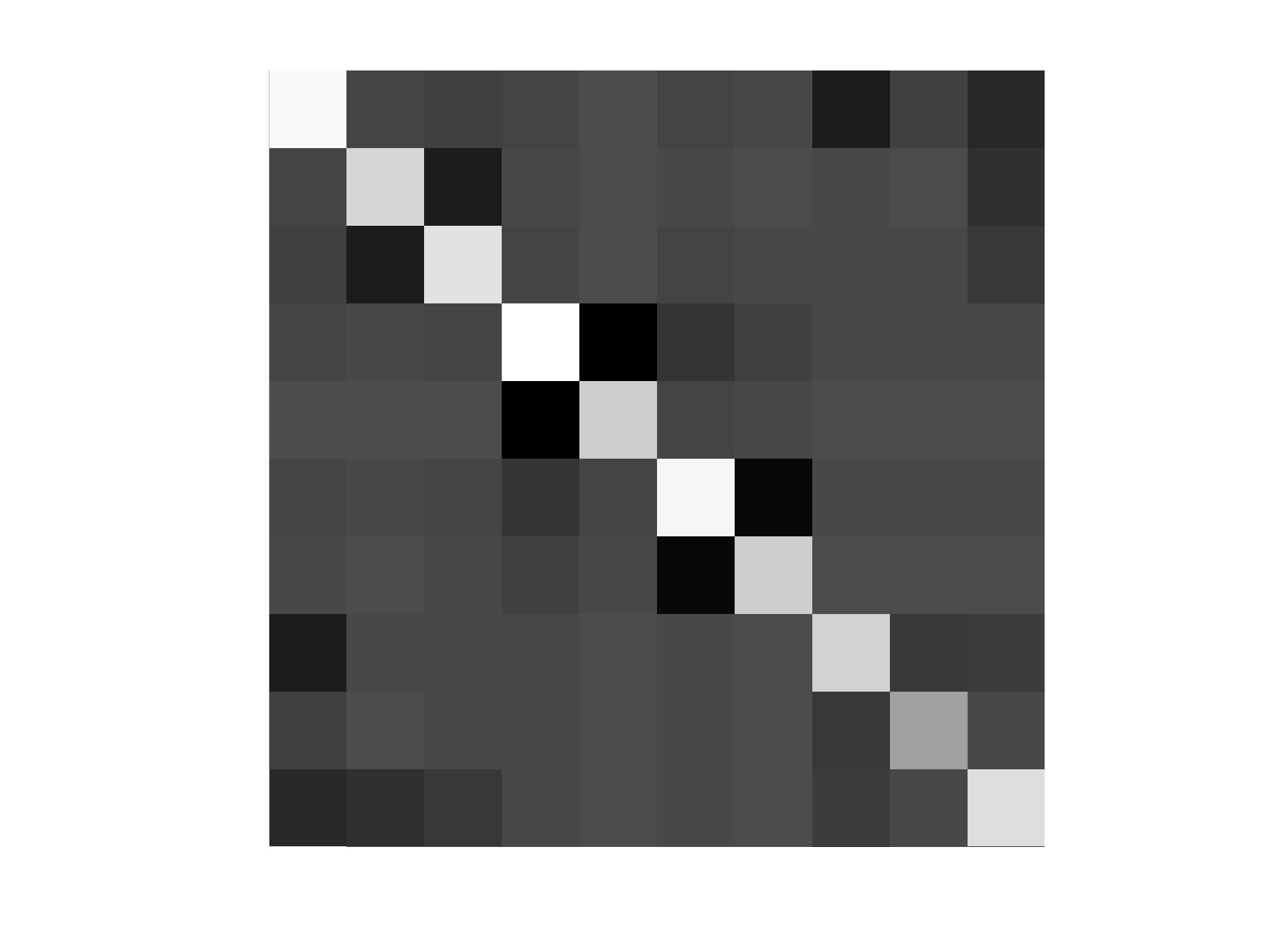}}
      \caption{Gray scale pattern of susceptance of all links of power network, i.e., $L$}
    \label{Generatorpat}
\end{figure}
Furthermore, additional plots are presented in Figure \ref{fig:freq_res} to show the similarity of the frequency behavior of the sparsely-controlled system to that of the $\text{LQR}$-controlled system. The upper left sub-figure, i.e., Figure \ref{fig:Sing_Range_rho0}, depicts the largest and smallest singular values of $\mathcal{S}$ and $\hat{\mathcal{S}}$ for the case of $\rho_{\text{\it rel}}= 0 \%$. It can be seen that the smallest singular values of the systems match for almost the whole frequency range and largest singular values achieve the same values for higher frequencies. Similar plots for the case of uncertain system with $\rho_{\text{\it rel}}= 30 \%$ are depicted in Figure \ref{fig:Sing_Range_rho5}. The plots depict that the deviation of the maximum singular value, caused by increasing the magnitude of the uncertainties, is much larger compared to the deviation of the minimum singular value. Also, the plots of Schatten 2-norm of the systems $\mathcal{S}$ and $\hat{\mathcal{S}}$, are depicted in lower sub-figures of \ref{fig:freq_res} for both cases, i.e., $\rho_{\text{\it rel}}= 0 \%$ and $\rho_{\text{\it rel}}= 30 \%$.  It is noteworthy that in neither of the cases, does the sparsification process seem to affect the higher frequency content of the closed-loop systems. This is desirable, since the controller sparsification will not be amplifying the harmonics in power grids, which are the main cause of power quality degradation.

In order to verify the relationship between the magnitude of the susceptance of each link, visualized in Figure \ref{Generatorpat}, and the density level of the corresponding entries in the controller design, we consider all cases with the $\rho_{\text{\it rel}}= 30 \%$ uncertainty on one link at a time, which results in $45$ cases. We then, compute ${f(K_{\theta})}$ and ${f(K_{\omega})}$, the sub-blocks of the controller matrix $K$, in which the matrix-valued function $f(X)=[f(X)_{ij}]$ is defined as
\begin{align*}
f(X)_{ij}=\left\{\begin{array}{lr}
 \|x_{ii}\|_0 + \|x_{ij}\|_0 + \|x_{ji}\|_0 + \|x_{jj}\|_0 &  \text{if~} i \neq j ,\\
 0 & \text{otherwise.}\end{array}\right.
\end{align*}
${f(K_{\theta})}$ and ${f(K_{\omega})}$ are used to visualize the number of controller links, necessary to be added to the generators connected with the uncertain link. Figures \ref{UncertaintyCardinalitya} and \ref{UncertaintyCardinalityb} show this visualization.

As depicted in Figs. \ref{UncertaintyCardinalitya} and \ref{UncertaintyCardinalityb}, in the case of links with higher susceptance, more communication links in controller design need to be established. This can be interpreted as the effective uncertainty of each link being proportional to the susceptance of that link. Therefore, an increase in susceptance of a link magnifies the uncertainty of that link, which results in establishment of more links in the designed controller to compensate for the fragility of the network on that link. This leads to similar patterns in Figures \ref{Generatorpat}, \ref{UncertaintyCardinalitya}, and \ref{UncertaintyCardinalityb}.


\begin{figure}[t]
\centering
\vspace{-.2cm}
    \subfloat[\label{UncertaintyCardinalitya}]{%
      \includegraphics[trim = 38mm 10mm 28mm 10mm, clip,width=0.4\textwidth]{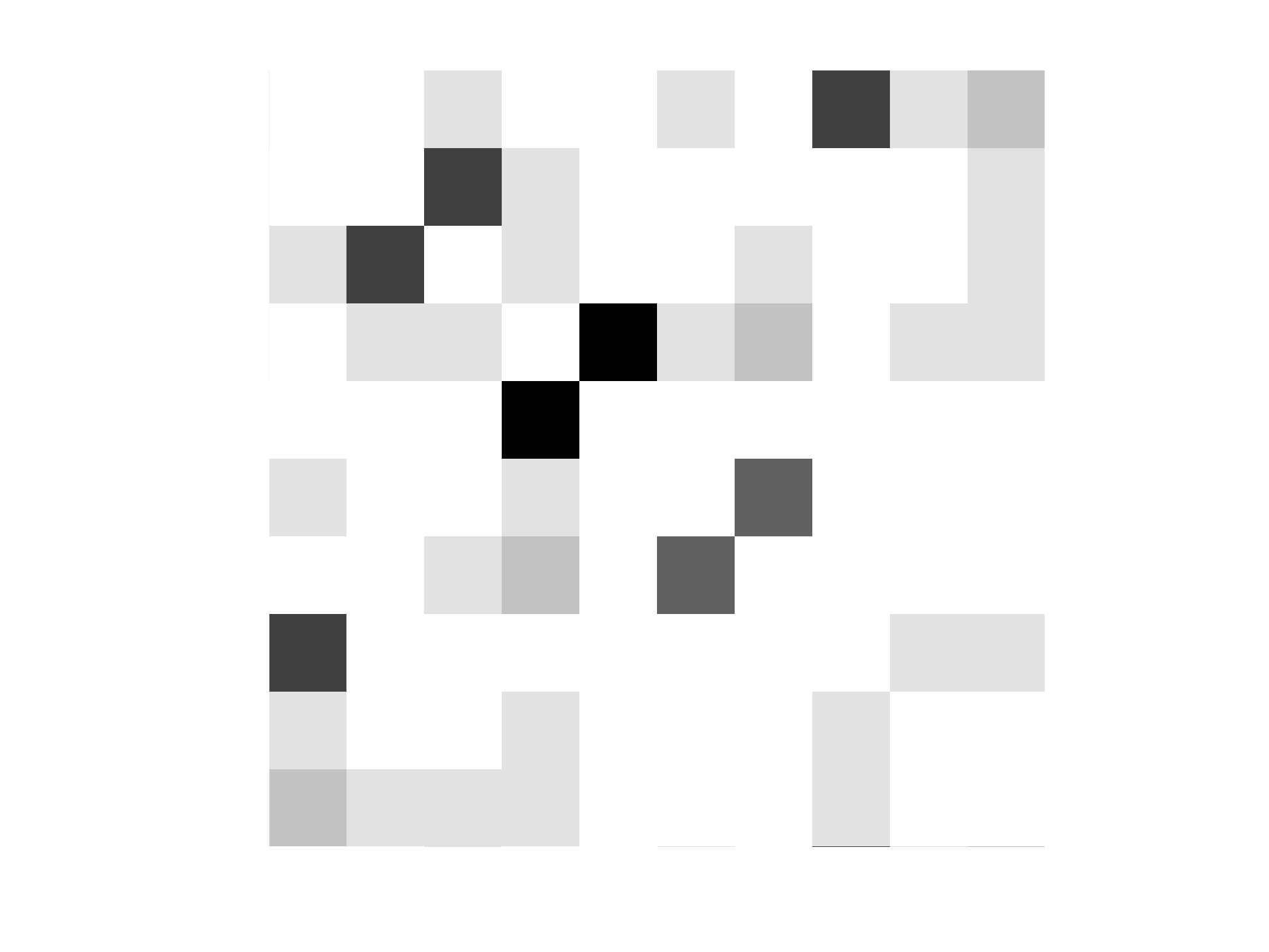}
    }
    \subfloat[\label{UncertaintyCardinalityb}]{%
      \includegraphics[trim = 38mm 10mm 28mm 10mm, clip, width=0.4\textwidth]{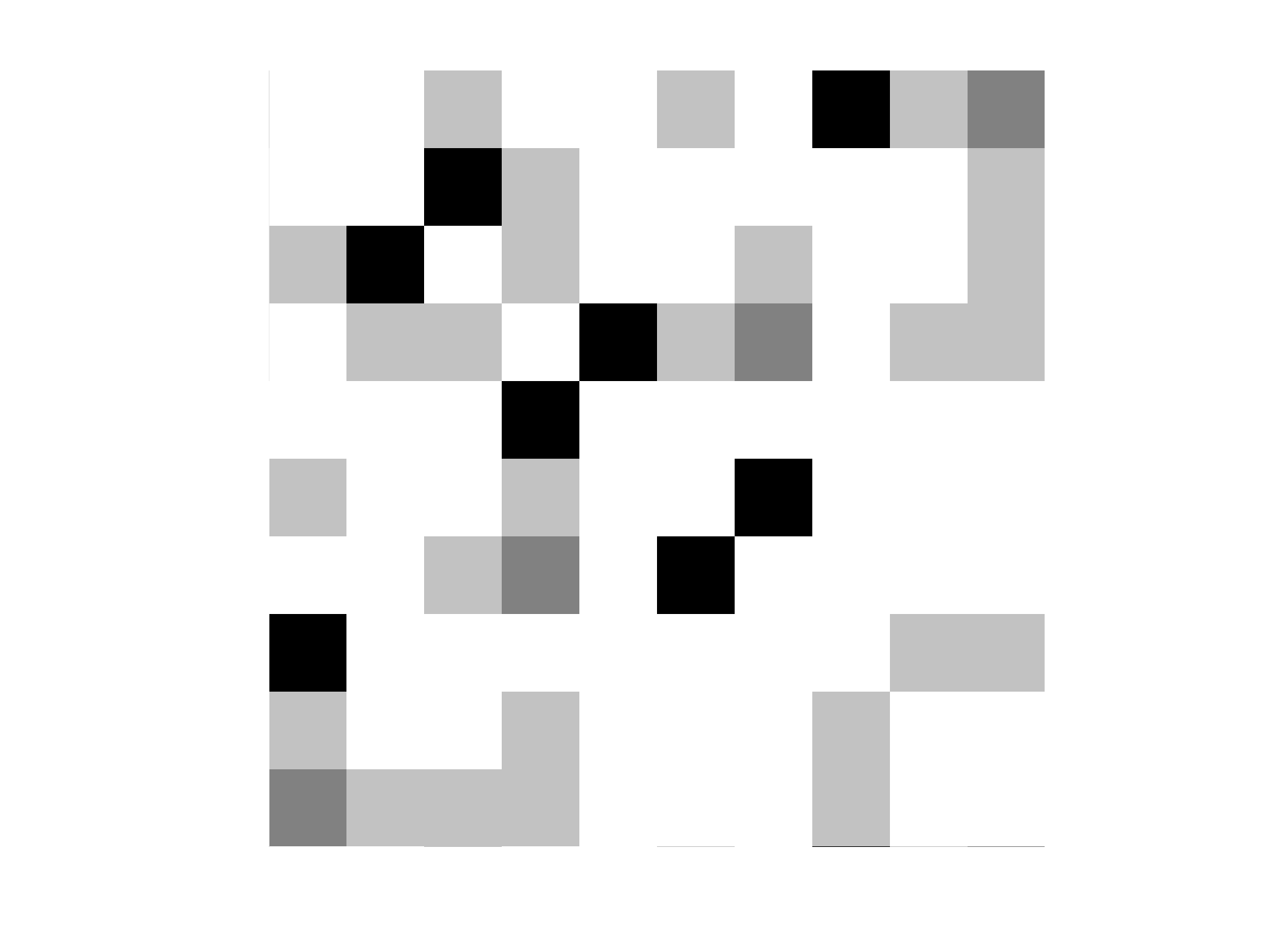}
      }
      \caption{ (a) Gray scale Pattern of $f(K_{\theta})+f(K_{\omega})$ for $\rho_{\text{\it rel}}=30 \%$ (b) Gray scale Pattern of $f(|K_{\theta}|+|K_{\omega}|)$ for $\rho_{\text{\it rel}}= 30 \%$}
    \label{UncertaintyCardinality}
\end{figure}



We furthermore showcase the effect of increasing the relative uncertainty of the network links on the cardinality of their corresponding controller entries for two randomly chosen links, connecting generator $4$ to generator $5$ and generators $2$ to $3$. As seen in figures $\ref{UncerCarda}$ and $\ref{UncerCardb}$, the increase of relative uncertainty, leads to construction of more communication links between two corresponding generators in the designed controller gain.

\begin{figure}[t]
\centering
\vspace{-.2cm}
    \subfloat[\label{UncerCarda}]{%
      \includegraphics[trim = 5mm 5mm 5mm 5mm, clip,width=0.45\textwidth]{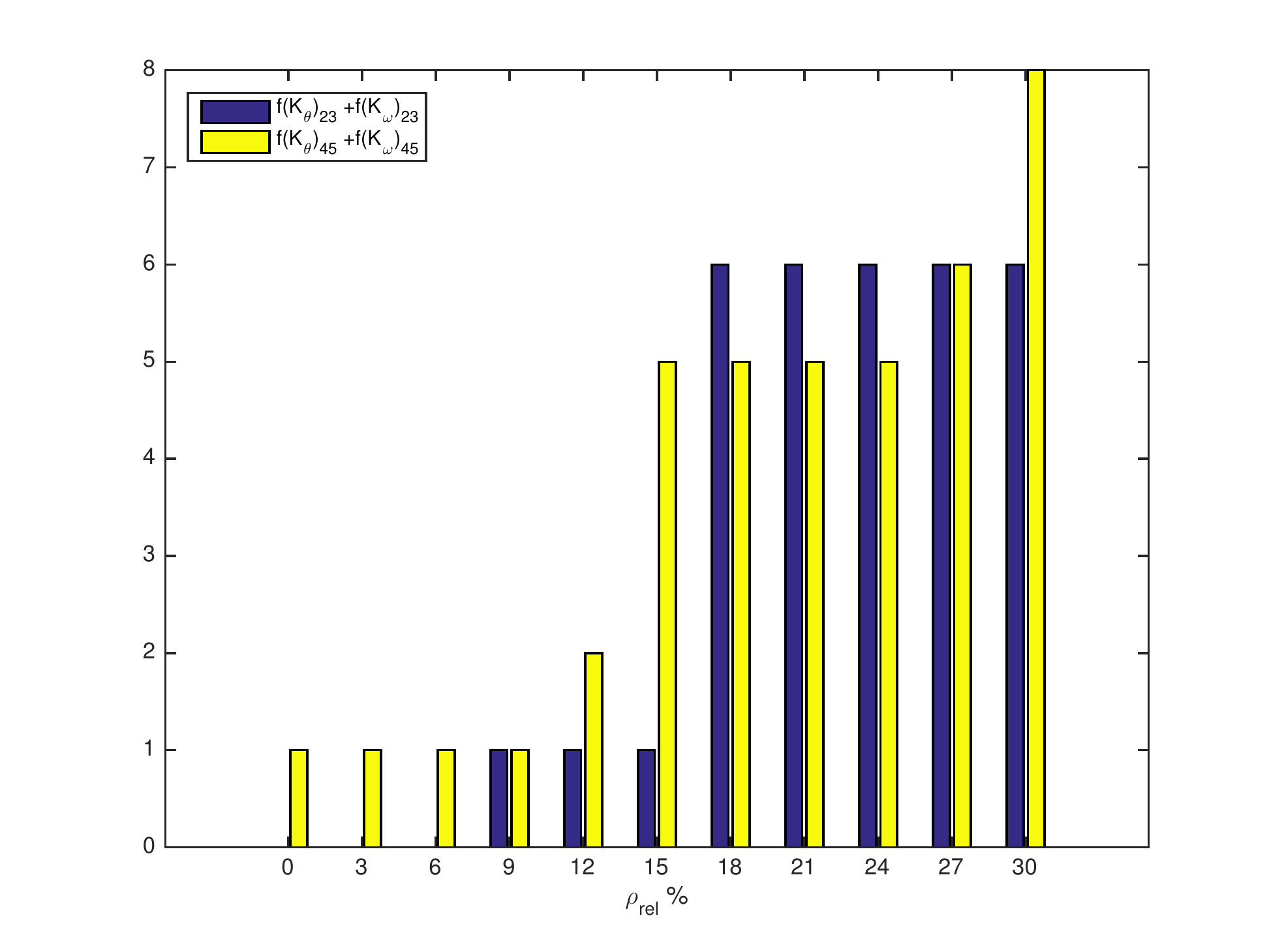}
    }
    \subfloat[\label{UncerCardb}]{%
      \includegraphics[trim = 5mm 5mm 5mm 5mm, clip, width=0.45\textwidth]{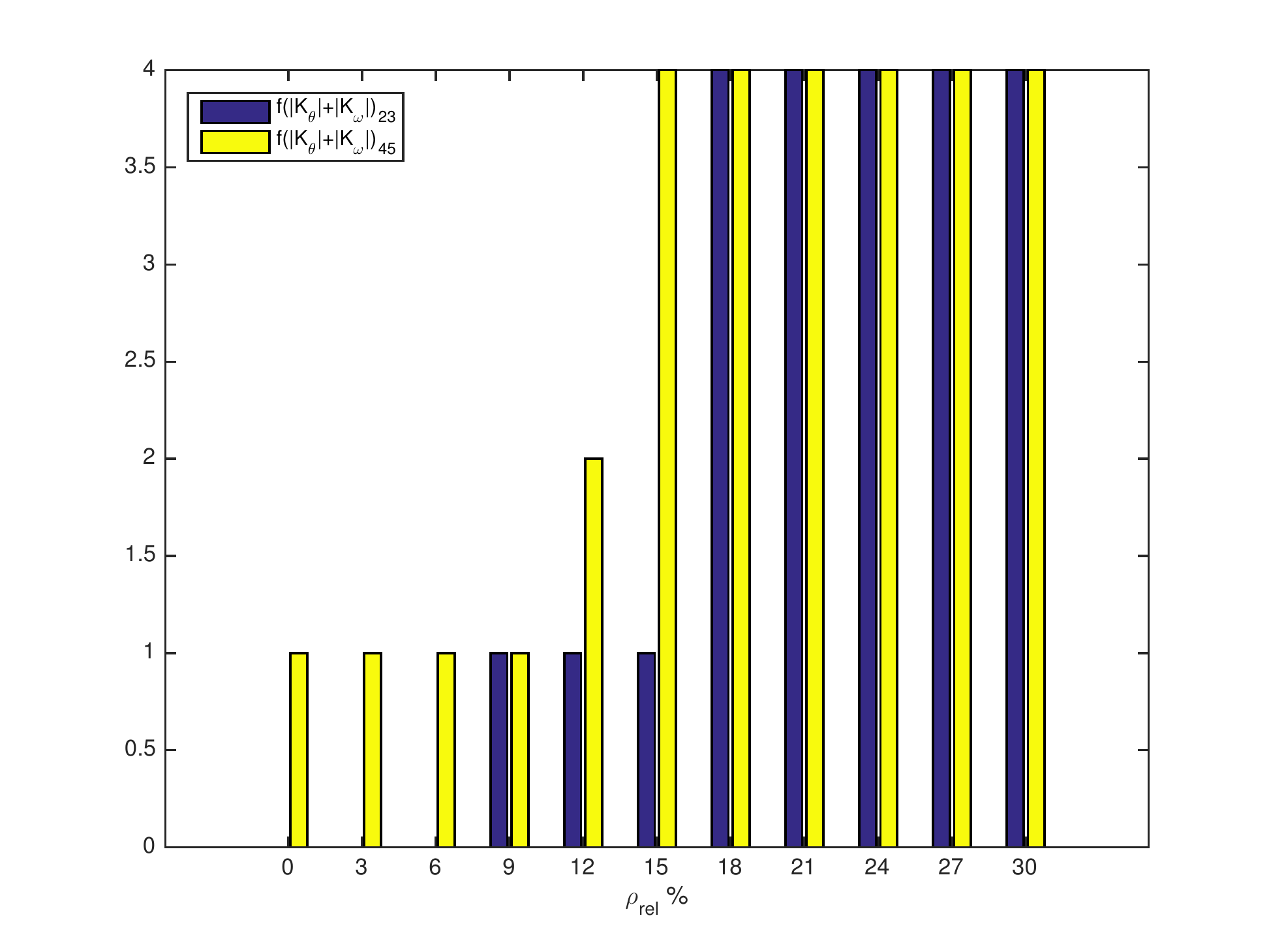}
      }
      \caption{ (a) $f(K_{\theta})+f(K_{\omega})$ vs $\rho_{\text{\it rel}} \%$ (b) $f(|K_{\theta}|+|K_{\omega}|)$ vs $\rho_{\text{\it rel}} \%$.}
    \label{UncerCard}
\end{figure}

\vspace{-6pt}
\section{Discussion}\label{sec:conc}
\vspace{-2pt}
We have explored a new optimization framework for the design of optimal sparse controllers for LTI systems under parametric uncertainties. The idea entails pruning the links of a centralized controller towards a sparse controller, while heeding the performance deterioration caused by this sparsification process. The design procedure is built upon constructing an optimization problem which seeks a sparse structured controller capable of exhibiting similar time and frequency characteristics of the previously designed controller, in the sense of $\mathcal{H}_2$ and $\mathcal{H}_{\infty}$ norms. We then, propose a computationally tractable algorithm, utilizing a bi-linear rank penalizing technique, to sub-optimally solve a fixed-rank reformulation of the aforementioned optimization problem. It should be noted one of the two steps in the bi-linear optimization algorithm, namely \emph{$Y$-minimization} step, has an analytical solution, which immensely enhances the run-time of our algorithm. Also, the outstanding performance of our proposed algorithm has been demonstrated through our extensive numerical simulations, run against various types of networks and systems. Although proving the global convergence of this algorithm still remains an open problem, we have shown that the sequence generated by the optimal cost at each iteration is monotonically decreasing and as a result, convergent.

One future research direction would encompass modifying our method to study the effect of the structure and magnitude of the uncertainties on the robustness of the closed-loop systems as well as the sparsity level of the controllers. An important application of this study would be the analysis of the robustness of networks, such as power grids, against possible attacks on the critical nodes.

\begin{figure}[t]
\centering
  \vspace{-.1cm} 
    \subfloat[\label{fig:Sing_Range_rho0}]{%
      \includegraphics[trim = 0mm 5mm 0mm 10mm, clip,width=0.45\textwidth]{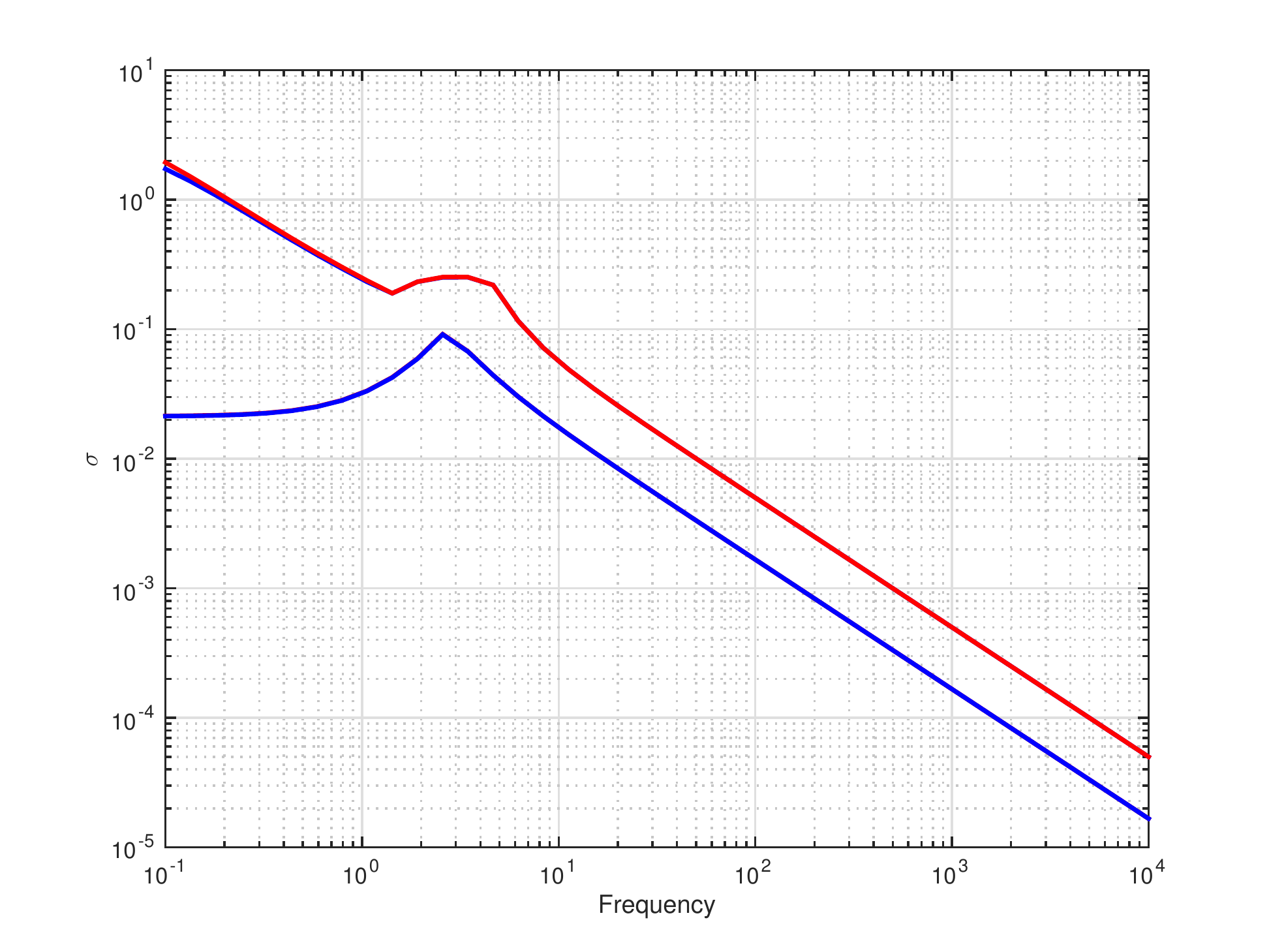}
    }
    \subfloat[\label{fig:Sing_Range_rho5}]{%
      \includegraphics[trim = 0mm 5mm 0mm 10mm, clip,width=0.45\textwidth]{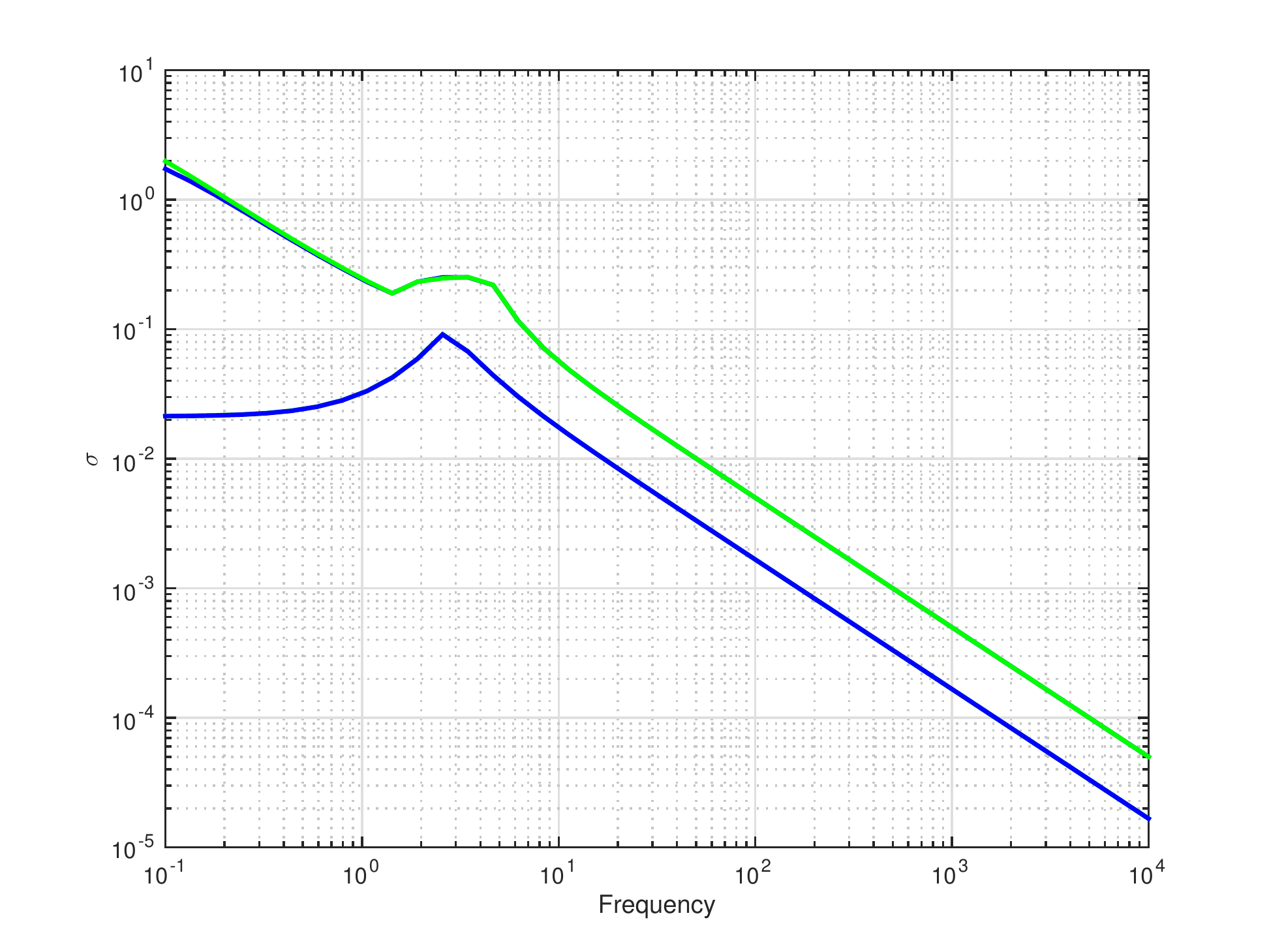}
    }\\
    \subfloat[\label{fig:schat_rho0}]{%
      \includegraphics[trim = 0mm 5mm 0mm 10mm, clip,width=0.45\textwidth]{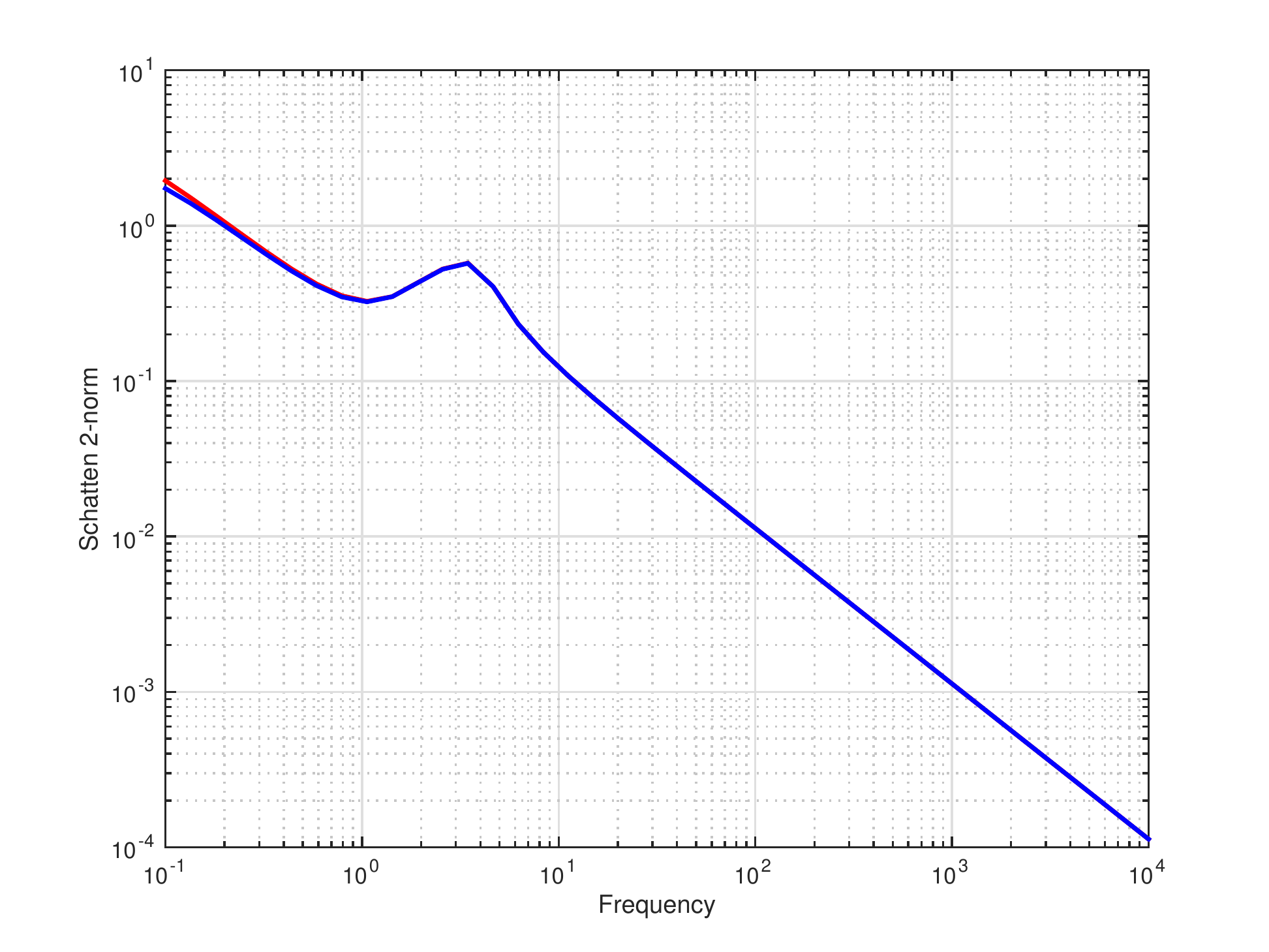}
    }
    \subfloat[\label{fig:schat_rho5}]{%
      \includegraphics[trim = 0mm 5mm 0mm 10mm, clip,width=0.45\textwidth]{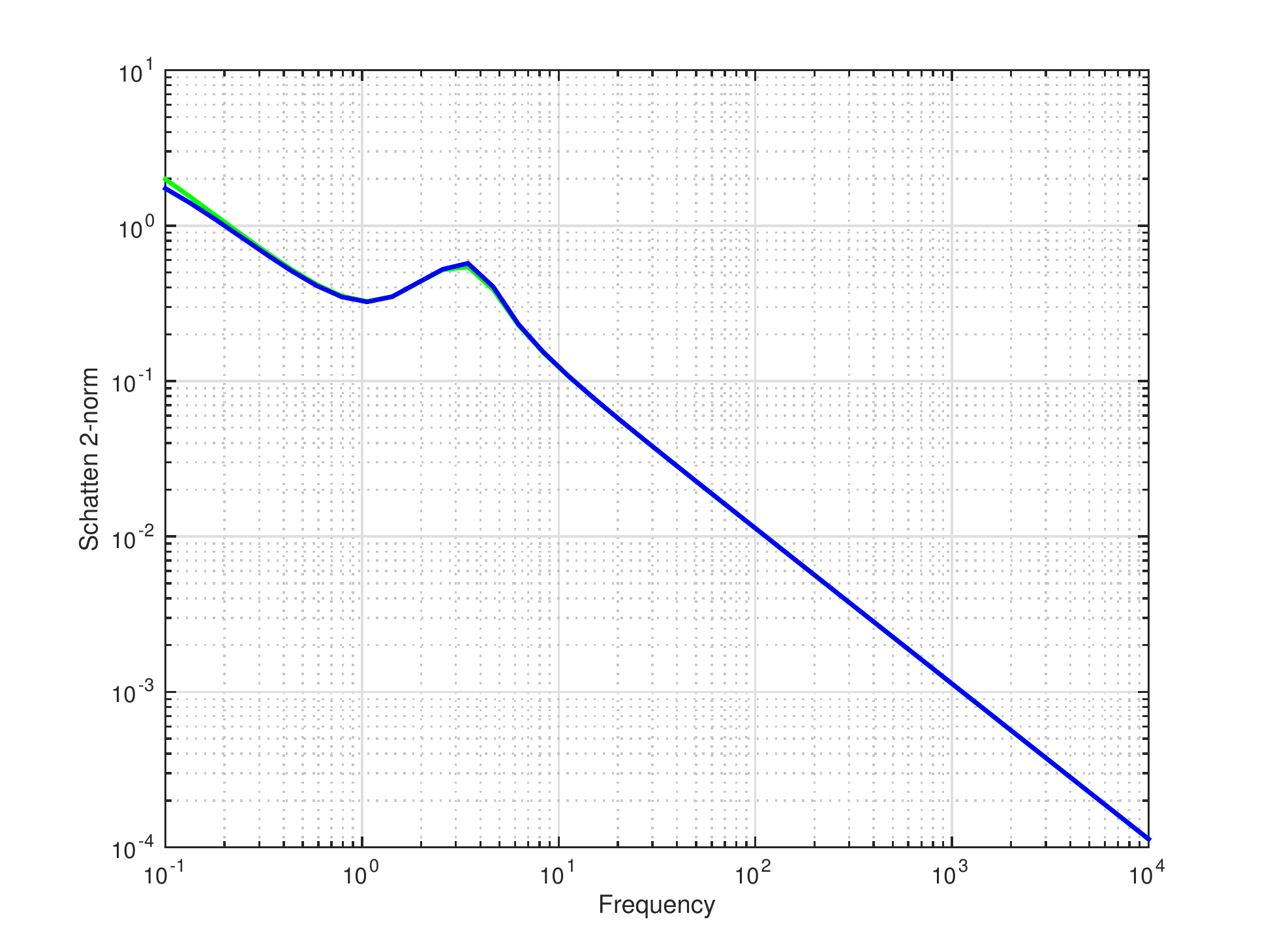}
    }
      \caption{Frequency characteristics of the closed-loop systems controlled by the $\text{LQR}$ (blue), the sparse controller (red) for the case $\rho_{\text{\it rel}}= 0 \%$, and the sparse controller (green) for the case $\rho_{\text{\it rel}}= 30 \%$. (a) and (b) depict maximum and minimum singular values for the cases of $\rho_{\text{\it rel}}= 0 \%$ and $\rho_{\text{\it rel}}= 30 \%$, respectively. (c) and (d) exhibit the Schatten 2-norm of the closed-loop system ((c) case $\rho_{\text{\it rel}}= 0 \%$ (d) case $\rho_{\text{\it rel}}= 30 \%$)}   \label{fig:freq_res}
\end{figure}


\bibliographystyle{IEEEtran}
\bibliography{main.bbl}


%
%

\end{document}